\documentclass{amsart}

\newtheorem{thm}{Theorem}[section]
\newtheorem{prop}[thm]{Proposition}
\newtheorem{cor}[thm]{Corollary}
\newtheorem{lem}[thm]{Lemma}
\newtheorem{rem}[thm]{Remark}
\newtheorem{conj}[thm]{Conjecture}

\numberwithin{equation}{section}

\def \vtk{V^{\otimes k}}

\def \bY{\bar Y}
\def \D{\Delta_k}
\def \Z{\mathbb Z}
\def \C{\mathbb C}

\def \N{\mathbb N}
\def \wt{{\rm wt}}
\def \Res{{\rm Res}}
\def \End{{\rm End}}

\def \mod{{\rm mod}}
\def \o{\omega}
\def \l{\lambda}

\newcommand{\ZZ}[0]{{\Z}_+}
\newcommand{\Lx}[0]{x^{j + 1} \frac{\partial}{\partial x}}
\newcommand{\Lo}[0]{x \frac{\partial}{\partial
x}}

\begin{document}

\title[Odd order permutation-twisted tensor product VOSA-modules]{Twisted modules for tensor product vertex operator superalgebras and permutation automorphisms of odd order}

\author{Katrina Barron}
\address{Department of Mathematics, University of Notre Dame}
\email{kbarron@nd.edu}
%\thanks{}

%    General info

\date{October 5, 2013}


\begin{abstract}
We construct and classify $(1 \; 2\; \cdots \; k)$-twisted $V^{\otimes k}$-modules for $k$ odd and for $V$ a vertex operator superalgebra. This extends previous results of the author, along with Dong and Mason, classifying all permutation-twisted modules for tensor product vertex operator algebras, to the setting of vertex operator superalgebras for odd order permutations.   We show why this construction does not extend to the case of permutations of even order in the superalgebra case and how the construction and classification in the even order case is fundamentally different than that for the odd order permutation case.  We present a conjecture made by the author and Nathan Vander Werf concerning the classification of permutation twisted modules for permutations of even order.
\end{abstract}

\maketitle

\section{Introduction}

Let $V$ be a  vertex operator (super)algebra, and for a fixed positive integer $k$,  consider the tensor product vertex operator (super)algebra $\vtk$ (see \cite{FLM3}, \cite{FHL}).  Any 
element $g$ of the symmetric group $S_k$ acts in a natural way on $\vtk$ as a vertex operator (super)algebra automorphism, and thus it is appropriate  to consider $g$-twisted $\vtk$-modules.  This is the setting for permutation orbifold conformal field theory, and for permutation orbifold superconformal field theory if the vertex operator superalgebra is not just super, but is also supersymmetric, i.e. is a representation of a Neveu-Schwarz super-extension of the Virasoro algebra.

In \cite{BDM}, the author along with Dong and Mason constructed and classified the $g$-twisted $V^{\otimes k}$-modules for $V$ a vertex operator algebra and $g$ any permutation.  In the present paper, we extend these results to $g$-twisted $V^{\otimes k}$-modules for $V$ a vertex operator superalgebra and $g$ a permutation of odd order.  In addition, we show that the results of \cite{BDM} for permutation-twisted tensor product vertex operator algebras and the results of the current paper do not extend in a straightforward way to the vertex operator superalgebra setting for permutations in the full symmetric group, but that for even order permutations the construction is necessarily fundamentally different.  

Twisted vertex operators were discovered and used in \cite{LW}.   Twisted modules for vertex operator algebras arose in the work of I. Frenkel, J. Lepowsky and A. Meurman \cite{FLM1}, \cite{FLM2}, \cite{FLM3} for the case of a lattice vertex operator algebra and the lattice isometry $-1$, in the course of the construction of the moonshine module vertex operator algebra (see also \cite{Bo}). This structure came to be understood as an ``orbifold model" in the sense of conformal field theory and string theory.  Twisted modules are the mathematical counterpart of ``twisted sectors", which are the basic building blocks of orbifold models in conformal field theory and string theory (see \cite{DHVW1}, \cite{DHVW2}, \cite{DFMS}, \cite{DVVV}, \cite{DGM}, as well as \cite{KS}, \cite{FKS}, \cite{Ba1}, \cite{Ba2}, \cite{BHS}, \cite{dBHO}, \cite{HO}, \cite{GHHO}, \cite{Ba3} and \cite{HH}).  Orbifold theory plays an important role in conformal field theory and in super extensions, and is also a way of constructing a new vertex operator (super)algebra from a given one.  

Formal calculus arising {}from twisted vertex operators associated to a an even lattice was systematically developed in \cite{Le1}, \cite{FLM2}, \cite{FLM3} and \cite{Le2}, and the twisted Jacobi identity was formulated and shown to hold for these operators (see also \cite{DL2}).  These results led to the introduction of the notion of $g$-twisted $V$-module \cite{FFR}, \cite{D}, for $V$ a vertex operator algebra and $g$ an automorphism of $V$.  This notion records the properties of twisted operators obtained in \cite{Le1}, \cite{FLM1}, \cite{FLM2}, \cite{FLM3} and \cite{Le2}, and provides an axiomatic definition of the notion of twisted sectors for conformal field theory.  In general, given a vertex operator algebra $V$ and an automorphism $g$ of $V$, it is an open problem as to how to construct a $g$-twisted $V$-module.

The focus of this paper is the study of permutation-twisted sectors for vertex operator superalgebras.    A theory of twisted operators for integral lattice vertex operator superalgebras and finite automorphisms that are lifts of a lattice isometry were studied in \cite{DL2} and \cite{Xu}, and the general theory of twisted modules for vertex operator superalgebras was developed by Li in \cite{Li2}.  Certain specific examples of permutation-twisted sectors in superconformal field theory have been studied from a physical point of view in, for instance, \cite{FKS}, \cite{BHS}, \cite{Maio-Schellekens1}, \cite{Maio-Schellekens2}. 

The main result of this paper is the explicit construction and classification of twisted sectors for permutation orbifold theory in the general setting of $V$ a vertex operator superalgebra and $g$ a cyclic permutation of odd order acting on $V^{\otimes k}$.  In particular, for $g$ a $k$-cycle for $k$ odd, and $V$ any vertex operator superalgebra, we show that the categories of weak, weak admissible and ordinary $g$-twisted $\vtk$-modules are isomorphic to the categories of weak, weak admissible and ordinary $V$-modules, respectively. (The definitions of weak, weak admissible and ordinary twisted modules are given in  Section \ref{twisted-module-definition-section}.)  To construct the isomorphism between the category of weak $g$-twisted $\vtk$-modules and the category of weak $V$-modules for $g$ a $k$-cycle, we explicitly define a weak $g$-twisted $\vtk$-module structure on any weak $V$-module.

We show that this method of constructing the permutation-twisted modules for tensor product vertex operator superalgebras fails for cycles of even length (and thus for permutations of even order in general) by pointing out that the operators one would use to construct the twisted modules in the case of a cycle of even order live in the wrong space;  see Remark \ref{wrong-space-remark}.   We show that these operators rather belong to a class of vertex operators which produce ``generalized twisted modules" which satisfy a more general Jacobi identity such as those given by the relativized twisted operators studied in \cite{DL2}; see Remark \ref{generalized-remark}.  

We make note of a conjecture made in \cite{BV-fermion} by the author and Vander Werf for permutation-twisted modules in the even order case.  This conjecture is based on our observations of the nature of the twisted modules for free fermions.   In particular, we conjecture
that for $g$ a $k$-cycle for $k$ even, and $V$ any vertex operator superalgebra,  the categories of weak, weak admissible and ordinary $g$-twisted $\vtk$-modules are isomorphic to the categories of weak, weak admissible and ordinary parity-twisted $V$-modules, respectively.
Our results can be used to construct permutation-twisted modules for a permutation $g$ of odd order acting on the multifold tensor product of a vertex operator superalgebra.

This paper is organized as follows. In Section \ref{definitions-section}, we recall the definitions of vertex operator superalgebra, and weak, weak admissible, and ordinary twisted module, as well as some of their properties.  In Section \ref{Delta-section}, we define the operator $\D(x)$ on a vertex operator superalgebra  $V$ and prove several important properties of $\D(x)$ which are needed in subsequent sections.   This is the main operator from which our twisted vertex operators will be built.   The main ideas for the proofs of these identities come {}from the development of this operator in the nonsuper case in \cite{BDM} and the supergeometry developed in \cite{Barron-memoirs}, \cite{Barron-CCM2003} and \cite{Barron-change} restricted to the vertex operator superalgebra setting.  

In Section  \ref{tensor-product-setting-section}, we develop the setting for $(1 \; 2 \; \cdots \; k)$-twisted $V^{\otimes k}$-modules and study the vertex operators for a $V$-module modified by the orbifolding $x \rightarrow x^{1/k}$ and composing with the operator $\Delta_k(x)$.  In particular we derive the supercommutator formula for these operators showing that these operators satisfy the twisted Jacobi identity for odd vectors if and only if $k$ is an odd integer.  
In Section \ref{tensor-product-twisted-construction-section}, we use the operators to define a weak $g = (1 \; 2 \;  \cdots \; k)$-twisted $\vtk$-module structure on any weak $V$-module in the case when  $k$ is odd.  As a result we construct a functor $T_g^k$ {}from the category of weak $V$-modules to the category of weak $g$-twisted $\vtk$-modules such that $T_g^k$ maps weak admissible (resp., ordinary) $V$-modules into weak admissible (resp., ordinary) $g$-twisted $\vtk$-modules.  In addition, $T_g^k$ preserves irreducible objects.

In Section \ref{classification-section}, we define a weak $V$-module structure on any weak $g = (1 \; 2 \; \cdots k)$-twisted $\vtk$-module, for $V$ a vertex operator superalgebra and $k$ odd.  In so doing, we construct a functor $U_g^k$ {}from the category of weak $g$-twisted $\vtk$-modules to the category of weak $V$-modules such that $T_g^k \circ U_g^k = id$ and $U_g^k \circ T_g^k = id$.  

In Section \ref{counterexample-section}, we comment on recent work and future work on constructing and classifying permutation-twisted modules for permutations of even order.  In particular, we present a conjecture from \cite{BV-fermion} as to what the classification of $(1 \; 2 \; \cdots \; k)$-twisted $V^{\otimes k}$-modules for $k$ even is, and we make note of the constructions and observations given in \cite{BV-fermion} for the case of $V$ the free fermion vertex operator superalgebra as evidence in support of this conjecture.

\section{Vertex operator superalgebras, twisted modules and some of their properties}\label{definitions-section}

In this section we recall some of the formal calculus we will need, and we recall the notions of vertex superalgebra and vertex operator superalgebra.  We also recall some properties of such structures, and prove a general geometrically inspired identity.  Then we present the notion of  $g$-twisted module for a vertex operator superalgebra and an automorphism $g$.  

\subsection{Formal calculus}

Let $x, x_0, x_1, x_2,$ etc., denote commuting independent formal variables.
Let $\delta (x) = \sum_{n \in \Z} x^n$.  We will use the binomial
expansion convention, namely, that any expression such as $(x_1 -
x_2)^n$ for $n \in \C$ is to be expanded as a formal power series in
nonnegative integral powers of the second variable, in this case
$x_2$.

For $r \in \mathbb{C}$ we have
\begin{equation}\label{delta-function1}
x_2^{-1}\left(\frac{x_1-x_0} {x_2}\right)^{r}
\delta\left(\frac{x_1-x_0} {x_2}\right) =
x_1^{-1}\left(\frac{x_2+x_0} {x_1}\right)^{-r}\delta
\left(\frac{x_2+x_0} {x_1}\right) ,
\end{equation}
and it is easy to see that for $k$ a positive integer,
\begin{equation}\label{delta-function2}
\sum_{p=0}^{k-1}\left(\frac{x_1-x_{0}}{x_2}\right)^{p/k}
x_2^{-1}\delta\left(\frac{x_1-x_0}{x_2}\right)=
x_2^{-1}\delta\Biggl(\frac{(x_1-x_0)^{1/k}}{x_2^{1/k}}\Biggr).
\end{equation}
Therefore, we have the $\delta$-function identity
\begin{equation}\label{delta-function3}
x_2^{-1} \delta \Biggl( \frac{(x_1 - x_0)^{1/k}}{x_2^{1/k}} \Biggr) = 
x_1^{-1} \delta \Biggl( \frac{(x_2 + x_0)^{1/k}}{x_1^{1/k}} \Biggr).
\end{equation}

We also have the three-term $\delta$-function identity 
\begin{equation}\label{three-term-delta}
x_{0}^{-1}\delta\left(\frac{x_1-x_2}{x_{0}}\right)-x_{0}^{-1}
\delta\left(\frac{x_2-x_1}{-x_{0}}\right)=x_2^{-1}\delta\left(\frac{x_1-x_0}{x_2}\right).
\end{equation}

Let $R$ be a ring, and let $O$ be an invertible linear operator on
$R[x, x^{-1}]$.  We define another linear operator $O^{\Lo}$ by
\[O^{\Lo} \cdot x^n = O^n x^n \]
for any $n \in \Z$.  For example, since the formal variable $z^{1/k}$ 
can be thought of as an invertible linear multiplication operator  
from $\C [x, x^{-1}]$ to $\mathbb{C}[z^{1/k},z^{-1/k}]
[x,x^{-1}]$, we have the corresponding operator $z^{(1/k) \Lo}$ {}from $\C 
[x,x^{-1}]$ to $\mathbb{C}[z^{1/k},z^{-1/k}]
[x,x^{-1}]$.  Note that $z^{(1/k) \Lo}$ can  be extended to a 
linear operator on $\C [[x,x^{-1}]]$ in the obvious way.

From Proposition 2.1.1 in \cite{H}, we have the following lemma.
\begin{lem}\label{Huang-lemma} (\cite{H})
For any formal power series in $f(x) \in x \mathbb{C} [[x]]$, given by 
\[f(x) = \sum_{j \in \N} a_j x^{j+1}  \ \ \ \mbox{for $a_j \in \C$} \]
there exists a unique sequence $\{A_j\}_{j \in \Z_+}$  in $\C$ such that 
\begin{equation}
f(x) = \exp \left( \sum_{j \in \Z_+} A_j x^{j+1} \frac{d}{dx} \right) a_0^{x \frac{d}{dx}} x.
\end{equation}
\end{lem}

Let $\varphi$ be a formal anti-commuting variable, that is, commuting with $x$ but satisfying $\varphi^2 = 0$.  From Proposition 3.5 in \cite{Barron-memoirs}, we have the following lemma.
\begin{lem}(\cite{Barron-memoirs})  For any formal power series in $f(x) \in x \mathbb{C} [[x]]$ given as in Lemma \ref{Huang-lemma}, and a choice of $a_0^{1/2}$, we have that
\begin{equation}\label{super-f}
\exp \left( -\sum_{j \in \Z_+} A_j L_j(x, \varphi) \right) (a_0^{1/2})^{-2L_0(x, \varphi)}  (x, \varphi) = \left(f(x) , \varphi \sqrt{f'(x)} \right)
\end{equation}
where 
\[L_j(x, \varphi) = - \left( x^{j+1} \frac{\partial}{\partial x} + \left(\frac{j+1}{2}\right) \varphi x^j \frac{\partial}{\partial \varphi} \right) ,\]
and $\sqrt{f'(x)}$ is the unique series in $\mathbb{C}[[x]]$ satisfying $\sqrt{f'(x)}^2 = f'(x)$ and such that $\sqrt{f'(x)} |_{x = 0} = a_0^{1/2}$.
\end{lem}

Note that the formal power series in $x \mathbb{C}[[x]], \varphi \mathbb{C}[[x]]$ given by (\ref{super-f}) is superconformal in the sense of \cite{Barron-memoirs}, but is not the most general form of a superconformal formal power series vanishing at zero.  In particular, the infinitesimal superconformal transformations involved, in this case $L_j(x,\varphi)$, only give a representation of the Virasoro algebra and not a super extension of the Virasoro algebra such as the Neveu-Schwarz algebra.  This reflects the fact that throughout this work, we will be assuming that we are working with a vertex operator superalgebra, i.e. one that has supercommuting properties with respect to a $\Z_2$-grading as well as a $\frac{1}{2} \Z$-grading, but is not necessarily supersymmetric, meaning it will not necessarily be a representation of any super extension of the Virasoro algebra.

\subsection{Vertex superalgebras, vertex operator superalgebras, and some of their properties}

A {\it vertex superalgebra} is a vector space which is $\Z_2$-graded (by {\it sign})
\begin{equation}
V= V^{(0)} \oplus V^{(1)}
\end{equation}
equipped with a linear map
\begin{eqnarray}
V &\longrightarrow& (\mbox{End}\,V)[[x,x^{-1}]]\\
v &\mapsto& Y(v,x)= \sum_{n\in\Z}v_nx^{-n-1} \nonumber
\end{eqnarray}
such that $v_n \in (\mathrm{End} \; V)^{(j)}$ for $v \in V^{(j)}$, $j \in \Z_2$, 
and equipped with a distinguished vector ${\bf 1} \in V^{(0)}$, (the {\it vacuum vector}), satisfying the following conditions for $u, v \in
V$:
\begin{eqnarray}
u_nv & \! = \! & 0\ \ \ \ \ \mbox{for $n$ sufficiently large};\\
Y({\bf 1},x) & \! = \! & Id_V;\\
Y(v,x){\bf 1} & \! \in \! & V[[x]]\ \ \ \mbox{and}\ \ \ \lim_{x\to 0}Y(v,x){\bf 1}\ = \ v; 
\end{eqnarray}
and for $u,v \in  V$ of homogeneous sign, the {\it Jacobi identity} holds 
\begin{equation*}
x^{-1}_0\delta\left(\frac{x_1-x_2}{x_0}\right) Y(u,x_1)Y(v,x_2) - (-1)^{|u||v|}
x^{-1}_0\delta\left(\frac{x_2-x_1}{-x_0}\right) Y(v,x_2)Y(u,x_1)
\end{equation*}
\begin{equation}
= x_2^{-1}\delta \left(\frac{x_1-x_0}{x_2}\right) Y(Y(u,x_0)v,x_2)
\end{equation}
where $|v| = j$ if $v \in V^{(j)}$ for $j \in \Z_2$.

This completes the definition. We denote the vertex superalgebra just
defined by $(V,Y,{\bf 1})$, or briefly, by $V$.

Note that as a consequence of the definition, we have that there exists a distinguished endomorphism $T \in (\mathrm{End} \; V)^{(0)}$ defined by 
\[ T(v) = v_{-2} {\bf 1} \ \ \ \ \mbox{for $v \in V$} \]
such that 
\[ [T, Y(v,x)] = Y(T(v), x)  \ = \ \frac{d}{dx} Y(v,x),\]
(cf. \cite{LL}, \cite{Barron-alternate}, \cite{Barron-n2axiomatic}).

A {\it vertex operator superalgebra} is a vertex superalgebra with a distinguished vector $\omega\in V_2$ (the {\it conformal element}) satisfying the following conditions: 
\begin{equation}
[L(m),L(n)]=(m-n)L(m+n)+\frac{1}{12}(m^3-m)\delta_{m+n,0}c
\end{equation}
for $m, n\in \Z,$ where
\begin{equation}
L(n)=\omega_{n+1}\ \ \ \mbox{for $n\in \Z$, \ \ \ \ i.e.},\
Y(\omega,x)=\sum_{n\in\Z}L(n)x^{-n-2}
\end{equation}
and $c \in \mathbb{C}$ (the {\it central charge} of $V$);
\begin{equation}
T = L(-1) \ \ \ \ \mbox{i.e.}, \ \frac{d}{dx}Y(v,x)=Y(L(-1)v,x) \ \mbox{for $v \in V$}; 
\end{equation}
$V$ is $\frac{1}{2}\Z$-graded (by {\it weight}) 
\begin{equation}
V=\coprod_{n\in \frac{1}{2} \Z}V_n 
\end{equation}
such that 
\begin{eqnarray}
L(0)v & \! = \! & nv \ = \ (\mbox{wt}\,v)v \ \ \ \mbox{for $n \in  \frac{1}{2} \Z$ and $v\in V_n$}; \\
{\rm dim} \, V & \! < \! & \infty ;\\
V_n & \! = \! &  0 \ \ \ \ \mbox{for $n$ sufficiently negative};
\end{eqnarray}
and $V^{(j)} = \coprod_{n\in \Z + \frac{j}{2}}V_n$ for $j \in \Z_2$.

This completes the definition. We denote the vertex operator superalgebra just
defined by $(V,Y,{\bf 1},\omega)$, or briefly, by $V$.

\begin{rem}\label{VOSAs-tensor-remark} {\em Note that if $(V, Y, {\bf 1})$ and $(V', Y', {\bf 1}')$ are two vertex superalgebras, then $(V \otimes V', \; Y \otimes Y', \; {\bf 1} \otimes {\bf 1}')$ is a vertex superalgebra where
\begin{equation}\label{define-tensor-product}
(Y \otimes Y') (u \otimes u', x) (v \otimes v') = (-1)^{|u'||v|} Y(u,x)v \otimes Y'(u',x)v'.
\end{equation}
If in addition, $V$ and $V'$ are vertex operator superalgebras with conformal vectors $\omega$ and $\omega'$ respectively, then $V\otimes V'$ is a vertex operator superalgebra with conformal vector $\omega \otimes {\bf 1}' + {\bf 1} \otimes \omega'$.}
\end{rem}

\begin{rem}\label{parity-grading-on-V}
{\em
As a consequence of the definition of vertex operator superalgebra, independent of the requirement that as a vertex superalgebra we should have $v_n \in (\mathrm{End} \, V)^{(|v|)}$, we have that $\mathrm{wt} (v_n u ) = \mathrm{wt} u + \mathrm{wt} v - n -1$, for $u,v \in V$ and $n \in \Z$.  This implies that $v_n \in (\mathrm{End} \, V)^{(j)}$ if and only if $v \in V^{(j)}$ for $j \in \mathbb{Z}_2$, without us having to assume this as an axiom.  
}
\end{rem}

Next we present some change of variables formulas and identities, generalizing \cite{H} in the nonsuper case, and reducing the results and techniques of \cite{Barron-memoirs}, \cite{Barron-CCM2003}, \cite{Barron-change} to the case of vertex operator superalgebras that are not necessarily supersymmetric.

Let $f_{t^{1/2}}(x) \in t^{-1}x \C[[x]]$ be a formal power series given by
\begin{equation}\label{define-H}
f_{t^{1/2}} (x) = \exp \Biggl( -\sum_{j \in \Z_+} A_j L_j(x, \varphi) \Biggr) (t^{-1/2}a_0^{1/2})^{-2L_0(x, \varphi)}  \cdot x
\end{equation}
for $A_j, a_0^{1/2} \in \C$ for $j \in \Z_+$, and for $t^{-1/2}$ a formal commuting variable.   In particular, we have  
\begin{eqnarray*}
\exp \Biggl( -\sum_{j \in \Z_+} A_j L_j(x, \varphi) \Biggr) (t^{-1/2}a_0^{1/2})^{-2L_0(x, \varphi)}  \cdot \varphi \biggr|_{x = 0} &=& (t^{-1/2}a_0^{1/2})^{-2L_0(x, \varphi)}  \cdot \varphi \\
& =& t^{-1/2} a_0^{1/2} \varphi.
\end{eqnarray*}
Define $\Theta_j =  \Theta_j(t^{-1/2} a_0^{1/2}, \{A_n\}_{n \in \Z_+}, x) \in \mathbb{C}[x][[t^{1/2}]]$, for $j \in \mathbb{N}$, by 
\begin{multline}\label{define-Theta}
f_{t^{1/2}} ( t a_0^{-1} w + f^{-1}_{t^{1/2}} (x)) - x \\
= \exp \Biggl( - \sum_{j \in Z_+} \Theta_j L_j(w, \rho) \Biggr) \exp(- \Theta_0 2L_0(w, \rho)) \cdot w 
\end{multline}
for $\rho$ an anticommuting formal variable.  That is, in particular, we define $\Theta_0$ such that 
\[\left. e^{-\Theta_0 2L_0(w, \rho)} \cdot \rho \right|_{w = 0} = e^{-\Theta_0} \rho.\]

By Corollary 3.44 of \cite{Barron-memoirs}, the formal series $\Theta_j$ are indeed well defined and in $\mathbb{C}[x][[t^{1/2}]]$.

\begin{prop}\label{Theta-prop}  Let $V= (V, Y, {\bf 1}, \omega)$ be a vertex operator superalgebra.  Then for $t^{1/2}$ a formal variable, $A_j \in \C$ for $j \in \Z_+$, and $a_0^{1/2} \in \C$, we have that 
\begin{multline}
e^{- \sum_{j \in \Z_+} A_j L(j)  } \cdot (t^{-\frac{1}{2}} a_0^\frac{1}{2})^{-2L(0)} Y(u,x) \cdot
(t^{-\frac{1}{2}} a_0^\frac{1}{2})^{2L(0)} \cdot e^{\sum_{j \in \Z_+} A_j L(j) } \\
= Y\biggl( (t^{-\frac{1}{2}} a_0^{\frac{1}{2}})^{-2L(0)} e^{- \sum_{j \in \Z_+} \Theta_j L(j) } \cdot e^{-2\Theta_0 L(0)} u, f_{t^{1/2}}^{-1}(x)
\biggr)  
\end{multline}
where 
\begin{equation}
f_{t^{1/2}}^{-1}(x) 
= (t^{-\frac{1}{2}} a_0^\frac{1}{2})^{-2x\frac{\partial}{\partial x}} \cdot
\exp \Bigl(-\sum_{j \in \Z_+} A_j x^{j+1} \frac{\partial}{\partial x} \Bigr) \cdot x
\end{equation} 
and the $\Theta_j = \Theta_j (t^{-1/2} a_0^{1/2}, \left\{A_n\right\}_{n \in \Z_+}, x)$, for $j \in \mathbb{N}$, are defined by (\ref{define-Theta}).  
\end{prop}

\begin{proof} By the proof of Equation (5.4.10)\footnote{There is a typo in this equation in
\cite{H}.  $A^{(0)}$ in the first line of equation (5.4.10) should be $A^{(1)}$ which is the infinite series $\{A_j^{(1)}\}_{j\in\Z_{+}}$,
where $A_j^{(1)}\in \C$.} in \cite{H} extended to the case of a vertex operator superalgebra or analogously by the proof of Equation (7.17) in \cite{Barron-CCM2003}, restricted to a not necessarily supersymmetric vertex operator superalgebra (i.e., by assuming all $G(j - 1/2)$ terms, for $j \in \Z$ are zero) the result follows.  
\end{proof}

\subsection{The notion of twisted module}\label{twisted-module-definition-section}

Let $(V, Y, {\bf 1})$ and $(V', Y', {\bf 1}')$ be vertex superalgebras.  A {\it homomorphism of vertex superalgebras} is a linear map $g: V \longrightarrow V'$ of $\Z_2$-graded vector spaces  such that $g({\bf 1}) = {\bf 1}'$ and 
\begin{equation}\label{automorphism}
g Y(v,x) =Y'(gv,x)g
\end{equation}
for $v\in V.$   Note that this implies that $g \circ T = T'\circ g$.  If in addition, $V$ and $V'$ are vertex operator superalgebras with conformal elements $\omega$ and $\omega'$, respectively, then a  {\it homomorphism of vertex operator superalgebras} is a homomorphism of vertex superalgebras $g$ such that $g(\omega) = \omega'$.  In particular $g V_n\subset V'_n$ for $n\in  \frac{1}{2} \mathbb{Z}$.

An {\it automorphism} of a vertex (operator) superalgebra $V$ is a bijective vertex (operator) superalgebra homomorphism from $V$ to $V$.

If $g$ is an automorphism of a vertex (operator) superalgebra $V$ such that $g$ 
has finite order, then $V$ is a direct sum of the eigenspaces $V^j$ of $g$,
\begin{equation}
V=\coprod_{j\in \Z /k \Z }V^j,
\end{equation}
where $k \in \mathbb{Z}_+$ and $g^k = 1$,  and
\begin{equation}
V^j=\{v\in V \; | \; g v= \eta^j v\},
\end{equation}
for $\eta$ a fixed primitive $k$-th root of unity.  We denote the projection of $v \in V$ onto the $j$-th eigenspace, $V^j$, by $v_{(j)}$.

Let $(V, Y, \mathbf{1})$ be a vertex superalgebra and $g$ an automorphism of $V$ of period $k$.  A {\em $g$-twisted $V$-module} is a $\mathbb{Z}_2$-graded vector space $M = M^{(0)} \oplus M^{(1)}$ equipped with a linear map 
\begin{eqnarray}
V &\rightarrow& (\End \, M)[[x^{1/k},x^{-1/k}]] \\  
v  &\mapsto& Y_g(v,x)=\sum_{n\in \frac{1}{k} \mathbb{Z} }v^g_n x^{-n-1}, \nonumber 
\end{eqnarray}
with $v_n^g \in (\mathrm{End} \; M)^{(|v|)}$, such that for $u,v\in V$ and $w\in M$ the following hold: 
\begin{equation}
v^g_nw=0 \mbox{ if $n$ is sufficiently large};
\end{equation}
\begin{equation} 
Y_g({\bf 1},x)=1;
\end{equation}
the twisted Jacobi identity holds: for $u, v\in V$ of homogeneous sign
\[x^{-1}_0\delta\left(\frac{x_1-x_2}{x_0}\right)
Y_g(u,x_1)Y_g(v,x_2)- (-1)^{|u||v|} x^{-1}_0\delta\left(\frac{x_2-x_1}{-x_0}\right)
Y_g(v,x_2)Y_g(u,x_1) \]
\begin{equation}\label{twisted-Jacobi}
= \frac{x_2^{-1}}{k} \sum_{ j \in \mathbb{Z}/k\mathbb{Z}} 
\delta\left( \eta^j \frac{(x_1-x_0)^{1/k}}{x_2^{1/k}}\right) Y_g(Y(g^ju,x_0)v,x_2).
\end{equation}

This completes the definition of $g$-twisted $V$-module for a vertex superalgebra.  We denote the  $g$-twisted $V$-module just defined by $(M, Y_g)$ or just by $M$ for short.  

We note that the generalized twisted Jacobi identity (\ref{twisted-Jacobi}) is equivalent to
\[x^{-1}_0\delta\left(\frac{x_1-x_2}{x_0}\right)
Y_g(u,x_1)Y_g(v,x_2)- (-1)^{|u||v|} x^{-1}_0\delta\left(\frac{x_2-x_1}{-x_0}\right)
Y_g(v,x_2)Y_g(u,x_1) \]
\begin{equation}\label{twisted-Jacobi-eigenspace}
= x_2^{-1}  \left(\frac{x_1-x_0}{x_2}\right)^{-r/k}
\delta\left( \frac{x_1-x_0}{x_2}\right) Y_g(Y(u,x_0)v,x_2)
\end{equation}
for $u \in V^r$, $r = 0, \dots, k-1$.  In addition, this implies that for $v\in V^r$, 
\begin{equation}\label{Y-for-eigenvector}
Y_g(v,x)=\sum_{n\in r/k + \Z }v^g_n x^{-n-1},
\end{equation}
and for $v \in V$,
\begin{equation}\label{limit-axiom}
Y_g(gv,x) = \lim_{x^{1/k} \to \eta^{-1} x^{1/k}} Y_g(v,x) .
\end{equation}

If $g = 1$, then a $g$-twisted $V$-module is a $V$-module for the vertex superalgebra $V$.  If $(V, Y, {\bf 1}, \omega)$ is a vertex operator superalgebra and $g$ is a vertex operator superalgebra automorphism of $V$, then since $\omega \in V^{0}$, we have that $Y_g(\o,x)$ has component operators which satisfy the Virasoro algebra relations and $Y_g(L(-1)u,x)=\frac{d}{dx}Y_g(u,x)$. In this case, a $g$-twisted $V$-module as defined above, viewed as a vertex superalgebra module, is called a {\em weak $g$-twisted $V$-module} for the vertex operator superalgebra $V$.

A {\em weak admissible}  $g$-twisted $V$-module is a weak $g$-twisted
$V$-module $M$ which carries a $\frac{1}{2k}\N$-grading
\begin{equation}\label{m3.12}
M=\coprod_{n\in\frac{1}{2k}\N}M(n)
\end{equation}
such that $v^g_mM(n)\subseteq M(n+\wt \; v-m-1)$ for homogeneous $v\in V$, $n \in \frac{1}{2k} \N$, and $m \in \frac{1}{k} \Z$.  We may assume that $M(0)\ne 0$ if $M\ne 0$.   If $g = 1$, then a weak admissible $g$-twisted $V$-module is called a weak admissible $V$-module.

An (ordinary) $g$-twisted $V$-module is a weak $g$-twisted $V$-module $M$ graded by $\C$ induced by the spectrum of $L(0).$ That is, we have
\begin{equation}\label{g3.14}
M=\coprod_{\lambda \in{\C}}M_{\lambda} 
\end{equation}
where $M_{\l}=\{w\in M|L(0)^gw=\l w\}$, for $L(0)^g = \omega_1^g$. Moreover we require that $\dim
M_{\l}$ is finite and $M_{n/2k +\l}=0$ for fixed $\l$ and for all
sufficiently small integers $n$.  If $g = 1$, then a $g$-twisted $V$-module is a $V$-module.

A {\it homomorphism of weak $g$-twisted $V$-modules}, $(M, Y_g)$ and $(M', Y_g')$, is a linear map $f: M \longrightarrow M'$ satisfying 
\begin{equation}
f(Y_g(v, x)w) = Y'_g(v, x) f(w).
\end{equation} 
for $v \in V$, and $w \in M$.  If in addition, $M$ and $M'$ are weak admissible $g$-twisted $V$-modules, then a {\it homomorphism of weak admissible $g$-twisted $V$-modules}, is a homomorphism of weak $g$-twisted $V$-modules such that 
$f(M(n)) \subseteq M'(n)$.  And if $M$ and $M'$ are ordinary $g$-twisted $V$-modules, then a {\it homomorphism of $g$-twisted $V$-modules}, is a homomorphism of weak  $g$-twisted $V$-modules such that 
$f(M_\lambda) \subseteq M'_\lambda$. 

The vertex operator superalgebra $V$ is called $g$-{\em rational} if every
weak admissible $g$-twisted $V$-module is completely reducible, i.e., a
direct sum of irreducible admissible $g$-twisted modules.

It was proved in \cite{DZ2} that if $V$ is $g$-rational then: (1) every
irreducible admissible $g$-twisted $V$-module is an ordinary
$g$-twisted $V$-module; and (2) $V$ has only finitely many isomorphism
classes of irreducible admissible $g$-twisted modules.

\begin{rem}{\em
In many works on vertex operator superalgebras, e.g. \cite{Li2}, \cite{DZ}, \cite{DZ2}, \cite{DH}, the condition that $v_n^g \in (\mathrm{End} \; M)^{(|v|)}$ for $v \in V$, is not given as one of the axioms of a module (twisted or untwisted) for a vertex superalgebra.  That is, it is not assumed that the $\mathbb{Z}_2$-grading of $V$ coincides with the $\mathbb{Z}_2$-grading of $M$ via the action of $V$ as super-endomorphisms acting on $M$.  Then in, for instance, \cite{DZ}, \cite{DZ2}, \cite{DH}, the notion of ``parity-stable module" is introduced for those modules that are representative of the $\mathbb{Z}_2$-grading of $V$.  However in \cite{BV-fermion}, we prove the following (reworded to fit our current setting):
\vspace{-.2in}
\begin{quote}
\begin{thm}\label{parity-stability-theorem}(\cite{BV-fermion})
Let $V$ be a vertex superalgebra and $g$ an automorphism.  Suppose $(M, Y_M)$ is a parity-unstable $g$-twisted $V$-module, i.e., such that all axioms except for $v_n^g \in (\mathrm{End} \; M)^{(|v|)}$ in the definition of a $g$-twisted $V$-module hold.  Then $(M, Y_M \circ \sigma_V)$ is a parity-unstable $g$-twisted $V$-module which is not isomorphic to $(M, Y_M)$.  Moreover $(M, Y_M)  \oplus (M, Y_M \circ \sigma_V)$ is a parity-stable $g$-twisted $V$-module, i.e. a $g$-twisted $V$ module in terms of the definition given above.   
\end{thm}
\end{quote}
Requiring weak twisted modules to be parity stable as part of the definition gives the more canonical notion of weak twisted module from a categorical point of view, for instance, for the purpose of defining a $V_1 \otimes V_2$-modules structure on $M_1 \otimes M_2$ for $M_j$ a $V_j$-module, for $j = 1,2$.  (See e.g. (\ref{define-tensor-product})). In particular, the notion of a weak $V$-module corresponding to a {\it representation} of $V$ as a vertex superalgebra only holds for parity-stable weak $g$-twisted $V$-modules, in that the vertex operators acting on a weak $g$-twisted $V$-module have coefficients in $\mathrm{End} \, M $ such that, the operators $v_m^g$ have a $\mathbb{Z}_2$-graded structure compatible with that of $V$.   For instance the operators $v_0^g$, for $v \in V$, give a representation of the Lie superalgebra generated by $v_0$ in $\mathrm{End} \, V$ if and only if $M$ is parity stable.  This corresponds to $V$ acting as endomorphisms in the category of vector spaces (i.e., via even or odd endomorphisms) rather than in the category of $\mathbb{Z}_2$-graded vectors spaces (i.e., as grade-preserving and thus strictly even endomorphisms).   However, it is interesting to note that, as is shown in \cite{BV-fermion}, for a lift of a lattice isometry, the twisted modules for a lattice vertex operator superalgebra naturally sometimes give rise to pairs of parity-unstable invariant subspaces that then must be taken as a direct sum to realize the actual twisted module.  }
\end{rem}

\section{The operator $\Delta_k (x)$}\label{Delta-section}
\setcounter{equation}{0} 

In this section, following and generalizing \cite{BDM}, we define an operator $\Delta_k(x)=\Delta_k^V(x)$ on a vertex operator superalgebra $V$ for a fixed positive integer $k$. In Section 3, we will use $\Delta_k(x)$ to construct a $g$-twisted 
$V^{\otimes k}$-module {}from a $V$-module where $g$ is a certain $k$-cycle with $k$ odd.

Let $\ZZ$ denote the positive integers.  Let $x$, $y$, $z$, and $z_0$ be formal variables commuting with each other.  Consider the polynomial
\[\frac{1}{k} (1 + x)^k - \frac{1}{k} \in x \C [x] . \]  
By Lemma \ref{Huang-lemma}, for $k \in \ZZ$, we can define $a_j \in \C$ for $j \in \ZZ$, by  
\begin{equation}\label{define-a}
\exp \Biggl( - \sum_{j \in \ZZ} a_j  \Lx
\Biggr) \cdot x = \frac{1}{k} (1 + x)^k -
\frac{1}{k} .
\end{equation}
For example, $a_1=(1-k)/2$ and $a_2=(k^2-1)/12.$

Let
\begin{eqnarray*}
f(x) &=&  z^{1/k} \exp \Biggl(- \!  \! \sum_{j \in \ZZ} a_j  \Lx \Biggr) 
\! \cdot x \\
&=& \exp \Biggl(- \! \! \sum_{j \in \ZZ} a_j  \Lx \Biggr) \! \cdot
z^{(1/k) \Lo} \cdot x\\
&=& \frac{z^{1/k}}{k} (1 + x)^k - \frac{z^{1/k}}{k} \quad \in
z^{1/k}x\mathbb{C}[x].
\end{eqnarray*}
Then the compositional inverse of $f(x)$ in $x\mathbb{C}[z^{-1/k}, 
z^{1/k}][[x]]$ is given by
\begin{eqnarray*}
f^{-1} (x) &=& z^{- (1/k) \Lo} \exp \Biggl( \sum_{j \in \ZZ} a_j  \Lx 
\Biggr) \cdot x \\
&=& z^{- 1/k} \exp \Biggl( \sum_{j \in \ZZ} a_j 
z^{- j/k} \Lx \Biggr) \cdot x \\
&=& (1 + k z^{- 1/k} x)^{1/k} - 1
\end{eqnarray*} 
where the last line is considered as a formal power series in
$z^{-1/k}x \mathbb{C}[z^{-1/k}][[x]]$, i.e., we are expanding
about $x = 0$ taking $1^{1/k} = 1$.

Now let $\Theta_j =  \Theta_j(z^{1/2k}, \{-a_n\}_{n \in \Z_+}, x) \in \mathbb{C}[x][[ z^{-1/2k}]]$, for $j \in \mathbb{N}$, where the $\Theta_j(t^{-1/2} a_0^{1/2}, \{A_n\}_{n \in \Z_+}, x)$ are defined by (\ref{define-Theta}), and the series $\{-a_j\}_{j \in \Z_+}$ are defined by (\ref{define-a}).  That is 
\begin{equation}\label{define-Theta-for-f}
f ( z^{-1/k} w + f^{-1} (x)) - x = e^{ \left( - \sum_{j \in Z_+} \Theta_j L_j(w, \rho) \right)} e^{- \Theta_0 2L_0(w, \rho)} \cdot w,
\end{equation}
and $e^{-\Theta_0 2L_0(w, \rho)} \cdot \rho = e^{\Theta_0} \rho$.

\begin{prop}\label{Theta-prop-BDM} The series
$\Theta_j (z^{1/2k}, \{- a_n \}_{n \in \ZZ},  \frac{1}{k}
z^{1/k - 1} z_0)$, for $j \in {\N}$, is a well-defined series with terms in
$\C[z_0][[z^{-1/k}]]$.  Furthermore
\begin{equation}\label{Theta-j}
\Theta_j \left(z^\frac{1}{2k}, \{- a_n \}_{n \in \ZZ}, k^{-1}
z^{\frac{1}{k} - 1} z_0\right) = - a_j (z + z_0)^{-\frac{ j}{k}} 
\end{equation} 
for $j \in \ZZ$, and
\begin{equation}\label{Theta-zero}
\exp \left(z^\frac{1}{2k}, \Theta_0 (\{- a_n \}_{n \in \ZZ},  k^{-1}
z^{\frac{1}{k} - 1} z_0) \right) = z^{ -\frac{1}{2k}(k-1)} (z +
z_0)^{\frac{1}{2k}(k-1)} , 
\end{equation} 
where $(z + z_0)^{- r/2k}$ is understood to be expanded in 
nonnegative integral powers of $z_0$.
\end{prop}

\begin{proof}  From (\ref{define-Theta-for-f}), the formal series $\Theta_j (z^{1/2k}, \{- a_n \}_{n \in \ZZ},  x)$ for $j \in \ZZ$ are the same as the series $\Theta_j (\{- a_n \}_{n \in \ZZ},  z^{1/k}, x)$ of \cite{BDM}.  We also have that the square of the exponential of  $\Theta_0 (z^{1/2k}, \{- a_n \}_{n \in \ZZ},  x)$ is equal to the exponential of the series $\Theta_0 (\{- a_n \}_{n \in \ZZ},  z^{1/k}, x)$ of \cite{BDM}.  Thus by Proposition 2.1 in \cite{BDM}, Equation (\ref{Theta-j}) follows.  In particular, in $w \C[z_0] [[z^{-1/k}]][[w]]$, we have 
\begin{eqnarray*} 
\lefteqn{\hspace{-.2in}
\exp \Biggl(\sum_{j \in \ZZ} a_j (z + z_0)^{-\frac{j}{k}} L_j(w, \rho)
\Biggr) z^{ \frac{1}{2k}(k-1)2L_0(w, \rho)} (z +
z_0)^{-\frac{1}{2k}(k-1)2L_0(w, \rho)} \cdot w  }  \\
&=& z^{\frac{1}{k} - 1} (z + z_0)^{-\frac{1}{k} + 1}
\exp \Biggl(- \! \sum_{j \in \ZZ} a_j (z + z_0)^{-\frac{j}{k}} w^{j+1} \frac{\partial}{\partial w}
\Biggr) \cdot w \\
&=& \left. f(f^{-1} (x) + z^{-1/k}w ) - x \right|_{x =\frac{1}{k} 
z^{1/k - 1} z_0}\\ 
&=&  \exp \Biggl( - \!\sum_{j \in \ZZ} \Theta_j 
(z^{1/2k}, \{- a_n \}_{n \in \ZZ},  x)L_j(w, \rho)  \Biggr) \\
& & \quad \left. \exp \left(-\Theta_0 (z^{1/2k}, \{- a_n \}_{n \in \ZZ}, x) 2L_0(w, \rho)\right)  \cdot w
\right|_{x =\frac{1}{k} z^{1/k - 1} z_0} .
\end{eqnarray*}

In addition, we have that in $\rho\C[z_0][[z^{-1/2k}]][[w]]$,
\begin{eqnarray*} 
\lefteqn{ \hspace{-.8in} \left.
\exp \Biggl(\sum_{j \in \ZZ} a_j (z + z_0)^{-\frac{j}{k}} L_j(w, \rho)
\Biggr) z^{ \frac{1}{2k}(k-1)2L_0(w, \rho)} (z +
z_0)^{-\frac{1}{2k}(k-1)2L_0(w, \rho)} \cdot \rho \right|_{w = 0} }  \\
&=&  z^{ -\frac{1}{2k}(k-1)} (z + z_0)^{\frac{1}{2k}(k-1)} \rho \\
&=&  \left. \exp \left(\Theta_0 (z^{1/2k}, \{- a_n \}_{n \in \ZZ}, x) \right)  \cdot \rho
\right|_{x =\frac{1}{k} z^{1/k - 1} z_0} ,
\end{eqnarray*}
giving Equation (\ref{Theta-zero}).  
\end{proof}

Let $V= (V, Y, {\bf 1}, \omega)$ be a vertex operator superalgebra.  Define the operator 
$\Delta_k^V(z) \in (\End \;V)[[z^{1/2k}, z^{-1/2k}]]$, by 
\[\Delta_k^V (z) = \exp \Biggl( \sum_{j \in \ZZ} a_j z^{- \frac{j}{k}} L(j) 
\Biggr) (k^\frac{1}{2})^{-2L(0)} \left(z^{\frac{1}{2k}\left( k-1\right)}\right)^{- 2L(0)} .\] 

\begin{prop}\label{psun1} 
In $(\End \;V)[[z^{1/2k}, z^{-1/2k}]]$, we have
\[\Delta_k^V (z) Y(u, z_0) \Delta_k^V (z)^{-1} = Y(\Delta_k^V 
(z + z_0)u, \left( z + z_0 \right)^{1/k} 
- z^{1/k}  ) ,\]
for all $u \in V$.
\end{prop}

\begin{proof} By Proposition \ref{Theta-prop}, where in our case, $A_j^{(1)}=-a_j$,
and by Proposition \ref{Theta-prop-BDM} above, we have that the steps of the proof of Proposition 2.2 in \cite{BDM} all hold in this setting.  That is, we have
\begin{eqnarray*}
\lefteqn{\Delta_k^V (z) Y(u, z_0) \Delta_k^V (z)^{-1}} \\
&=& \! \exp \Biggl(\sum_{j \in \ZZ} a_j z^{- \frac{j}{k}} L(j) \Biggr) 
Y( (k^\frac{1}{2})^{-2L(0)} \left(z^{\frac{1}{2k}\left( k-1\right)}\right)^{- 2L(0)}u, 
k^{-1} z^{ 1/k - 1}z_0)\cdot \\
& & \quad \cdot\exp \Biggl(- \sum_{j \in \ZZ}  a_j z^{- \frac{j}{k}} L(j) \Biggr) \\ 
&=& \! (z^\frac{1}{2k})^{2L(0)} \exp \Biggl(\sum_{j \in \ZZ} 
a_j L(j) \Biggr)  (z^\frac{1}{2k})^{-2L(0)}
Y( (k^\frac{1}{2})^{-2L(0)} z^{-\frac{1}{2k}( k-1)2L(0)}, \\
& & \quad k^{-1} z^{ 1/k - 1}z_0)  (z^\frac{1}{2k})^{2L(0)} \exp \Biggl( - 
\sum_{j \in \ZZ} a_j L(j) \Biggr)  (z^\frac{1}{2k})^{-2L(0)} \\
&=& \!  (z^\frac{1}{2k})^{2L(0)} Y \Biggl(  (z^\frac{1}{2k})^{-2L(0)}
\exp \Biggl( - \sum_{j \in \ZZ} \Theta_j( z^\frac{1}{2k}, \{- a_n \}_{n \in \ZZ}, k^{-1} z^{ \frac{1}{k} - 1}z_0  ) L(j) \Biggr) 
\cdot \Biggr. \\
& & \quad \exp \left(- \Theta_0 (z^\frac{1}{2k}, \{- a_n \}_{n \in \ZZ},  k^{-1}
z^{\frac{1}{k} - 1}z_0) 2L(0) \right)  (k^\frac{1}{2})^{-2L(0)} z^{-\frac{1}{2k}( k-1) 2L(0)}u, \\
& & \quad \Biggl. f^{-1} (k^{-1} z^{ 1/k - 1} z_0 ) \Biggr)  (z^\frac{1}{2k})^{-2L(0)}\\
&=& \!  (z^\frac{1}{2k})^{2L(0)} Y \Biggl( (z^\frac{1}{2k})^{-2L(0)}\exp 
\Biggl( \sum_{j \in \ZZ} a_j (z + z_0)^{- \frac{j}{k}} L(j) \Biggr) 
z^{ \frac{1}{2k}(k-1)2L(0)} 
\end{eqnarray*}
\begin{eqnarray*}
& & \quad (z +
z_0)^{-\frac{1}{2k}(k-1)2L(0)}
(k^\frac{1}{2})^{-2L(0)} z^{-\frac{1}{2k}( k-1)2L(0)}u, \left( 1 + z^{- 1}z_0 \right)^\frac{1}{k} - 1 
\Biggr)\\
& & \quad  (z^\frac{1}{2k})^{-2L(0)} \\
&=& \!  (z^\frac{1}{2k})^{2L(0)}  Y \left( (z^\frac{1}{2k})^{-2L(0)} \Delta_k^V (z + z_0) u, 
\left( 1 + z^{- 1}z_0 \right)^\frac{1}{k} - 1 \right) 
 (z^\frac{1}{2k})^{-2L(0)} \\
&=& \! Y \left(\Delta_k^V (z + z_0) u, 
\left( z + z_0 \right)^\frac{1}{k} - z^\frac{1}{k}  \right) 
\end{eqnarray*}
as desired. \end{proof}

Define $\Delta_k^{(x,\varphi)} (z) \in (\End \; \C[x, x^{-1}][\varphi])
[[z^{1/2k}, z^{-1/2k}]]$ by
\[\Delta_k^{(x, \varphi)} (z) = \exp \Biggl( \sum_{j \in \ZZ} a_j z^{- \frac{j}{k}} L_j(x, \varphi)
\Biggr) (k^\frac{1}{2})^{-2L_0(x, \varphi)} (z^{\frac{1}{2k}( k -1)})^{- 2L_0(x, \varphi)} ,\]
that is
\begin{eqnarray*} 
\Delta_k^{(x, \varphi)} (z) \cdot(x, \varphi) &=& \left(z^\frac{1}{k})^{L_0(x, \varphi)} \cdot (k z^{1 - \frac{1}{k}} f(x), \;  \varphi k^\frac{1}{2} z^{\frac{1}{2} - \frac{1}{2k}} \sqrt{f'(x)}\right) \\
&=& \left( (z^\frac{1}{k} + x)^k - z, \; \varphi k^\frac{1}{2} ( z^\frac{1}{k} + x)^{\frac{1}{2} (k-1)}\right).
\end{eqnarray*}
 
\begin{prop}
In $(\End \; \C[x, x^{-1}][\varphi]) [[z^{1/2k}, z^{-1/2k}]]$, we have 
\begin{eqnarray}
- \Delta_k^{(x,\varphi)} (z) \frac{\partial}{\partial x} + k^{-1}
z^{\frac{1}{k} - 1} \frac{\partial}{\partial x}
\Delta_k^{(x, \varphi)} (z)  &=& \frac{\partial}{\partial z} \Delta_k^{(x, \varphi)} (z) , 
\label{identity-in-rep}\\
- \Delta_k^{(x,\varphi)} (z)^{-1} \frac{\partial}{\partial x} + k 
z^{- \frac{1}{k} + 1} \frac{\partial}{\partial x}
\Delta_k^{(x,\varphi)} (z)^{-1}  &=& k z^{- \frac{1}{k} + 1} 
\frac{\partial}{\partial z} \Delta_k^{(x,\varphi)} (z)^{-1} . 
\label{second-identity-in-rep}
\end{eqnarray}
\end{prop}

\begin{proof}  Since $L_j(x, \varphi) \cdot x^n = -x^{j+1} \frac{\partial}{\partial x}\cdot x^n$ for $n \in \Z$, we have from Proposition 2.3 in \cite{BDM} that Equations (\ref{identity-in-rep}) and (\ref{second-identity-in-rep}) hold in $(\End \; \C[x, x^{-1}]) [[z^{1/k}, z^{-1/k}]] \subset (\End \; \C[x, x^{-1}][\varphi]) [[z^{1/2k}, z^{-1/2k}]]$.  

Next we observe that in $\C[x, x^{-1}] [\varphi][[z^{1/2k}, z^{-1/2k}]]$, we have
\begin{eqnarray*}
\lefteqn{
- \Delta_k^{(x,\varphi)} (z) \frac{\partial}{\partial x} \cdot \varphi + k^{-1}
z^{\frac{1}{k} - 1} \frac{\partial}{\partial x} \Delta_k^{(x,\varphi)} (z) \cdot \varphi } \\
&=& k^{-1} z^{\frac{1}{k} - 1}  \frac{\partial}{\partial x} \varphi k^\frac{1}{2} ( z^\frac{1}{k} + x)^{\frac{1}{2} (k-1)}\\
&=&  z^{\frac{1}{k} - 1} \varphi k^{-\frac{1}{2}}\frac{1}{2}(k-1) ( z^\frac{1}{k} + x)^{\frac{1}{2} (k-3)}\\
&=&  \frac{\partial}{\partial z}  \varphi k^{\frac{1}{2}}( z^\frac{1}{k} + x)^{\frac{1}{2} (k-1)}\\
&=& \frac{\partial}{\partial z} \Delta_k^x (z) \cdot \varphi.
\end{eqnarray*}

By Proposition 3.11 in \cite{Barron-memoirs} and Proposition 2.3 in \cite{BDM}, we have 
\begin{eqnarray*}
\Delta_k^{(x,\varphi)} (z) \cdot  (\varphi x^n) &=& (\Delta_k^{(x,\varphi)} (z) \cdot  \varphi)   (\Delta_k^{(x,\varphi)} (z) \cdot  x^n)  \\
 &=& (\Delta_k^{(x,\varphi)} (z) \cdot  \varphi)   (\Delta_k^{(x,\varphi)} (z) \cdot  x)^n  
\end{eqnarray*}
for all $n \in \Z$.  Therefore
\begin{eqnarray*}
\lefteqn{
- \Delta_k^{(x,\varphi)} (z) \frac{\partial}{\partial x} \cdot \varphi x^n  + 
k^{-1} z^{\frac{1}{k} - 1} \frac{\partial}{\partial x} 
\Delta_k^{(x,\varphi)} (z) \cdot \varphi x^n } \\
&=&  -n  \Delta_k^{(x, \varphi)} (z) \cdot \left( \varphi x^{n - 1}
\frac{\partial}{\partial x} \cdot x  \right) + k^{-1} z^{\frac{1}{k} - 1} 
\frac{\partial}{\partial x} (\Delta_k^{(x,\varphi)} (z) \cdot  \varphi)   (\Delta_k^{(x,\varphi)} (z) \cdot  x)^n  \\
&=&  -n \left( \Delta_k^{(x, \varphi)} (z) \cdot\varphi\right) \left(\Delta_k^{(x, \varphi)} (z) \cdot x\right)^{n - 1}  \left(\Delta_k^{(x,\varphi)} (z)
\frac{\partial}{\partial x} \cdot x \right) \\
& & \quad + k^{-1} z^{\frac{1}{k} - 1} 
\left( \frac{\partial}{\partial x} (\Delta_k^{(x,\varphi)} (z) \cdot  \varphi) \right)   (\Delta_k^{(x,\varphi)} (z) \cdot  x)^n \\
& & \quad + k^{-1} z^{\frac{1}{k} - 1} 
(\Delta_k^{(x,\varphi)} (z) \cdot  \varphi)  n (\Delta_k^{(x,\varphi)} (z) \cdot  x)^{n-1} \frac{\partial}{\partial x}   (\Delta_k^{(x,\varphi)} (z) \cdot  x) \\
&=& \left( - \Delta_k^{(x,\varphi)} (z) \frac{\partial}{\partial x} \cdot \varphi + k^{-1}
z^{\frac{1}{k} - 1} \frac{\partial}{\partial x} \Delta_k^{(x,\varphi)} (z) \cdot \varphi\right) (\Delta_k^{(x,\varphi)} (z) \cdot  x)^n  \\
& & \quad +  (\Delta_k^{(x,\varphi)} (z) \cdot  \varphi)  n  (\Delta_k^{(x,\varphi)} (z) \cdot  x)^{n-1} \left(   - \Delta_k^{(x,\varphi)} (z) \frac{\partial}{\partial x} \cdot x \right. \\
& & \quad \left. + k^{-1}
z^{\frac{1}{k} - 1} \frac{\partial}{\partial x} \Delta_k^{(x,\varphi)} (z) \cdot x    \right)\\
&=& \left(\frac{\partial}{\partial z}   (\Delta_k^{(x,\varphi)} (z) \cdot  \varphi)  \right) (\Delta_k^{(x,\varphi)} (z) \cdot  x)^n   \\
& & \quad +  (\Delta_k^{(x,\varphi)} (z) \cdot  \varphi)  n  (\Delta_k^{(x,\varphi)} (z) \cdot  x)^{n-1} \left(\frac{\partial}{\partial z} (\Delta_k^{(x,\varphi)} (z) \cdot  x) \right)\\
&=& \frac{\partial}{\partial z} \left(  (\Delta_k^{(x,\varphi)} (z) \cdot  \varphi)   (\Delta_k^{(x,\varphi)} (z) \cdot  x)^n \right) \\
&=& \frac{\partial}{\partial z} \left( \Delta_k^{(x, \varphi)} (z) \cdot \varphi x^n \right)\\
\end{eqnarray*}
for all 
$n \in \Z$.  Equation (\ref{identity-in-rep}) follows by linearity in $\C[x, x^{-1}] [\varphi][[z^{1/2k}, z^{-1/2k}]]$.

Similarly, noting that 
\[\Delta_k^{(x,\varphi)} (z)^{-1} (x, \varphi) = \left( (x + z)^{1/k} - z^{1/k}, \; \varphi k^{-\frac{1}{2}} (x + z)^{\frac{1}{2k}(1-k)} \right)\]
we have
\begin{eqnarray*}
\lefteqn{
- \Delta_k^{(x,\varphi)} (z)^{-1} \frac{\partial}{\partial x} \cdot \varphi + k 
z^{- \frac{1}{k} + 1} \frac{\partial}{\partial x} 
\Delta_k^{(x,\varphi)} (z)^{-1} \cdot \varphi } \\
&=&  k z^{- \frac{1}{k} + 1} \frac{\partial}{\partial x} \varphi k^{-\frac{1}{2}} (x + z)^{\frac{1}{2k}(1-k)}\\
&=&  z^{- \frac{1}{k} + 1} \varphi k^{-\frac{1}{2}} \frac{1}{2}(1-k) (x + z)^{\frac{1}{2k}(1-3k)}\\
&=& k z^{- \frac{1}{k} + 1} \frac{\partial}{\partial z} \varphi k^{-\frac{1}{2}} (x + z)^{\frac{1}{2k}(1-k)} 
\\
&=& k z^{- \frac{1}{k} + 1} \frac{\partial}{\partial z} 
\Delta_k^{(x,\varphi)} (z)^{-1} \cdot \varphi .
\end{eqnarray*}
The proof of identity (\ref{second-identity-in-rep}) acting on $\varphi x^n$ for 
$n \in \Z$ is analogous to the proof of identity 
(\ref{identity-in-rep}) on $\varphi x^n$ for $n \in \Z$.  Identity 
(\ref{second-identity-in-rep}) then follows by linearity.  
\end{proof}

Let $\mathfrak{L}$ be the Virasoro algebra with basis $L_j$, $j \in \Z$, and central charge $d \in \C$.  The identities (\ref{identity-in-rep}) and (\ref{second-identity-in-rep})  can be thought of as identities for the representation of the Virasoro algebra on $\C[x, x^{-1}][\varphi]$ given by $L_j \mapsto -L_j(x, \varphi)$, for $j \in \Z,$ with central charge  equal to zero.  We want to prove the corresponding identity for certain other representations of the Virasoro algebra, in 
particular for vertex operator superalgebras.  We do this by following the method of proof used in Chapter 4 of \cite{H}, and extending \cite{BDM}.  Let $\kappa^{1/2}$ be another formal commuting variable.  We first prove the identity in $\mathcal{ U}_\Pi (\mathfrak{L})  [[z^{1/2k}, z^{-1/2k}]][[\kappa^{1/2}, \kappa^{-1/2}]]$ where $\mathcal{ U}_\Pi 
(\mathfrak{L})$ is a certain extension of the universal enveloping algebra for the Virasoro algebra, and then letting $\kappa^{1/2} = k^{1/2}$, the identity will follow in $(\End \; V) [[z^{1/2k}, 
z^{-1/2k}]]$ where $V$ is a certain type of module for the Virasoro algebra.

We want to construct an extension of $\mathcal{ U}(\mathfrak{L})$, the 
universal enveloping algebra for the Virasoro algebra, in which the operators
$(\kappa^{1/2})^{-2L_0}$ and $(z^{1/2k})^{ (1 - k)2L_0}$ can be defined.  Let 
$V_\Pi$ be a vector space over $\mathbb{C}$ with basis $\{P_j : j \in 
\frac{1}{2}\mathbb{Z}\}$. Let $\mathcal{ T}(\mathfrak{L} \oplus V_\Pi)$ be the tensor 
algebra generated by the direct sum of $\mathfrak{L}$ and $V_\Pi$, and 
let $\mathcal{ I}$ be the ideal of $\mathcal{ T}(\mathfrak{L} \oplus V_\Pi)$ 
generated by
\begin{multline}
\lefteqn{\Bigl\{L_i \otimes L_j - L_j \otimes L_i - [L_i,L_j], \; L_i \otimes d - d 
\otimes L_i, \; P_r \otimes P_s - \delta_{rs} P_r,  \Bigr. }   \\
P_r \otimes d -
d \otimes P_r, \; \Bigl. P_r \otimes L_j - L_j \otimes P_{r + j} : \; i,j \in 
\mathbb{Z}, \; r,s \in \frac{1}{2} \Z \Bigr\}.
\end{multline}
Define $\mathcal{ U}_\Pi (\mathfrak{L}) = \mathcal{ T}(\mathfrak{L} \oplus V_\Pi)/
\mathcal{ I}$.  For any formal commuting variable $t^{1/2}$ and for $n\in\mathbb{Z}$, we define
\[(t^{1/2})^{n2L_0} = \sum_{j \in \frac{1}{2} \mathbb{Z}} P_j t^{nj} \in \mathcal{ U}_\Pi(\mathfrak{L})[[t^{1/2},t^{-1/2}]] .\]
Note then that $(\kappa^{1/2})^{-2L_0}$ and $(z^{1/2k})^{ (1 - k)2L_0}$ are well-defined elements of 
\[U_\Pi(\mathfrak{L})[[z^{1/2k},z^{-1/2k}]][[\kappa^{1/2}, \kappa^{-1/2}]].\]

In $\mathcal{ U}_\Pi (\mathfrak{L})[[z^{1/2k}, z^{-1/2k}]]
[[\kappa^{1/2}, \kappa^{-1/2}]]$, define 
\[\Delta_k^\mathfrak{L} (z) = \exp \Biggl( \sum_{j \in \ZZ} 
a_j z^{- j/k} L_j \Biggr) (\kappa^{1/2})^{-2L_0} (z^{1/2k})^{ (1 - k)2L_0}.\]

\begin{prop}\label{in-universal}
In $\mathcal{ U}_\Pi (\mathfrak{L})[[z^{1/2k}, z^{-1/2k}]]
[[\kappa^{1/2}, \kappa^{-1/2}]]$, we have
\begin{eqnarray}
\Delta_k^\mathfrak{L} (z) L_{-1} - \kappa^{-1} z^{\frac{1}{k} - 1} L_{-1} 
\Delta_k^\mathfrak{L} (z) 
&= &\frac{\partial}{\partial z}\Delta_k^\mathfrak{L} (z) , \label{identity-in-universal}\\
\Delta_k^\mathfrak{L} (z)^{-1} L_{-1} - \kappa z^{-\frac{1}{k} + 1} L_{-1} 
\Delta_k^\mathfrak{L} (z)^{-1} 
&=& k z^{-\frac{1}{k} + 1} \frac{\partial}{\partial z} \Delta_k^\mathfrak{L} (z)^{-1} . 
\label{second-identity-in-universal}
\end{eqnarray} 
\end{prop}

\begin{proof}  In $\mathcal{ U}_\Pi (\mathfrak{L})[[z^{1/2k}, 
z^{-1/2k}]][[\kappa^{1/2}, \kappa^{-1/2}]]$, we have
\begin{eqnarray*}
\lefteqn{
\Delta_k^\mathfrak{L} (z) L_{-1} - \kappa^{-1} z^{\frac{1}{k} - 1} L_{-1} 
\Delta_k^\mathfrak{L} (z) } \\
&=& \kappa^{-1} z^{\frac{1}{k} - 1} \left[ e^{\sum_{j \in \ZZ} 
a_j z^{- j/k} L_j}, L_{-1} \right]  (\kappa^{1/2})^{-2L_0} (z^{1/2k})^{ (1 - k)2L_0} \hspace{1in}
\end{eqnarray*}
\begin{eqnarray*}
&=&  \kappa^{-1} z^{\frac{1}{k} - 1} \sum_{n \in \ZZ} \frac{1}{n!}\Biggl( 
\sum_{j_1,...,j_n \in \ZZ} a_{j_1} \cdots a_{j_n}  z^{-(j_1 + \cdots + j_n)/k} 
\Biggr. \\ 
& & \ \ \Biggl. \biggl( \sum_{i = 1,...,n} L_{j_1} L_{j_2} \cdots
L_{j_{i-1}}[L_{j_i},L_{-1}] L_{j_{i + 1}} \cdots L_{j_n} \biggr)\!\Biggr)
 (\kappa^{1/2})^{-2L_0} (z^{1/2k})^{ (1 - k)2L_0}
\end{eqnarray*} 
which is a well-defined element of $\mathcal{ U}_\Pi (\mathfrak{L})
[[z^{1/2k}, z^{-1/2k}]][[\kappa^{1/2}, \kappa^{-1/2}]]$ involving 
only elements $L_j$ with $j \in {\N}$.  The right-hand side of 
(\ref{identity-in-universal}) also involves only $L_j$ for $j \in {\N}$.  
Thus comparing with the identity (\ref{identity-in-rep}) for the 
representation $L_j \mapsto - L_j(x, \varphi)$, the identity (\ref{identity-in-universal}) must hold.  The proof of (\ref{second-identity-in-universal}) is analogous.  \end{proof}

Let $V$ be a module for the Virasoro algebra satisfying $V = \coprod_{n 
\in \frac{1}{2} \mathbb{Z}} V_n$.  For $j \in \mathbb{Z}$, let $L(j) \in \End \; V$  
and $c \in \C$ be the representation images of $L_j$ and $d$, 
respectively, for the Virasoro algebra.  Assume that for $v \in V_n$, 
we have $L(0)v = nv$.  For any formal variable $t^{1/2}$, define $(t^{1/2})^{j2L(0)} 
\in (\mathrm{End} \; V)[[t^{1/2},t^{-1/2}]]$ by
\[(t^{1/2})^{j2L(0)}v = t^{jn}v \]
for $v \in V_n$.  Or equivalently, let $P(n) : V \rightarrow V_n$
be the projection {}from $V$ to the homogeneous subspace of weight $n$ for
$n \in \frac{1}{2} \mathbb{Z}$.  Then
\[(t^{1/2})^{j2L(0)}v = \sum_{n \in \mathbb{Z}} t^{jn} P(n) v \]
for $v \in V$.  The elements $P(n) \in \mathrm{End} \; V$ can be thought
of as the representation images of $P_n$ in the algebra $\mathcal{ U}_\Pi
(\mathfrak{L})$.

Note that for $k$ a positive integer, and $k^{1/2}$ a fixed square root of $k$, we have that  $(k^{1/2})^{-2L(0)}$ is a well-defined element of 
$\mathrm{End} \; V$ and for $z^{1/2k}$ a formal commuting variable, $(z^{1/2k})^{(1-k)2L(0)}$ is a well-defined element of $(\mathrm{End} \; V)[[z^{1/2k},z^{-1/2k}]]$.

In $(\End \; V)[[z^{1/2k}, z^{-1/2k}]]$, define 
\begin{equation}\label{Delta-for-a-module}
\Delta_k^V (z) = \exp \Biggl( \sum_{j \in \ZZ} a_j z^{- j/k} L(j) 
\Biggr) (k^\frac{1}{2})^{-2L(0)} (z^{\frac{1}{2k}})^{( 1 - k )2 L(0)} .
\end{equation}
{}From Proposition 4.1.1 in \cite{H} extended to this setting of a $\frac{1}{2} \Z$-graded $\mathcal{L}$-module, or equivalently, from Proposition 3.32 of \cite{Barron-memoirs} restricted to the Virasoro subalgebra of the N=1 Neveu-Schwarz algebra, and from Proposition 
\ref{in-universal}, we 
obtain the following corollary.

\begin{cor}\label{c2.5}
In $(\End \; V)[[z^{1/2k}, z^{-1/2k}]]$, we have
\begin{eqnarray}
\Delta_k^V (z) L(-1) - \frac{1}{k} z^{1/k - 1} L(-1) 
\Delta_k^V (z) \!
&=& \! \frac{\partial}{\partial z}\Delta_k^V (z) , \label{identity in voa}\\ 
\Delta_k^V (z)^{-1} L(-1) - k z^{- 1/k + 1} L(-1) 
\Delta_k^V (z)^{-1}  \! 
&=& \! k z^{- 1/k + 1} \frac{\partial}{\partial z}
\Delta_k^V (z)^{-1}  . \label{second identity in voa}
\end{eqnarray}
In particular, the identities hold for $V$ being any vertex operator
superalgebra. 
\end{cor}

\section{The setting of $(1 \; 2 \; \cdots k$)-twisted $V^{\otimes k}$-modules and the operators $Y_M(\Delta_k(x)u, x^{1/k})$ for a $V$-module $(M, Y_M)$}\label{tensor-product-setting-section}
\setcounter{equation}{0}

Now we turn our attention to tensor product vertex operator superalgebras.
Let $V=(V,Y,{\bf 1},\omega)$ be a vertex operator superalgebra, and let $k$ be a
fixed positive integer.  Then by Remark \ref{VOSAs-tensor-remark},  $V^{\otimes k}$ is also a
vertex operator superalgebra, and the permutation group $S_k$
acts naturally on $\vtk$ as signed automorphisms.  That is
$(j \; j+1) \cdot (v_1 \otimes v_2 \otimes \cdots \otimes v_k) = (-1)^{|v_j||v_{j+1}|} (v_1 \otimes v_2 \otimes \cdots v_{j-1} \otimes v_{j+1} \otimes v_j \otimes v_{j+2} \otimes \cdots \otimes v_k)$, and we take this to be a left action so that, for instance
\begin{eqnarray}
\qquad (1 \; 2 \cdots  k) : V \otimes V \otimes \cdots \otimes V \! \! \! \! &\longrightarrow & \! \! \! \! V \otimes V \otimes \cdots \otimes V\\
v_1 \otimes v_2 \otimes \cdots \otimes v_k \! \! \! \! \! \! & \mapsto & \! \! \! \! \!  \! (-1)^{|v_1|(|v_2| + \cdots + |v_k|)} v_2 \otimes v_3 \otimes \cdots \otimes v_k \otimes v_1. \nonumber
\end{eqnarray}
(Note that in \cite{BDM}, this action was given as a right action.  For convenience we make the change here to a left action as in \cite{BHL}).

Let $g=(1 \; 2 \; \cdots \; k)$.  In the next section we will, construct a functor $T_g^k$ {}from the category of weak  $V$-modules to the category of weak $g$-twisted modules for $\vtk$ for the case when $k$ is odd.  This construction will be based on the operators $Y_M(\Delta_k(x)u, x^{1/k})$ for a $V$-module $(M, Y_M)$ and for $\Delta_k(x)$ the operator on $V$ defined and studied in the previous section.  Thus in this section, we establish several properties of these operators.  In particular, we will also show why our construction of $g$-twisted $V^{\otimes k}$-modules will not follow through for $k$ even.  

For $k$ odd, we construct these weak $g$-twisted $V^{\otimes k}$-modules
by first defining $g$-twisted vertex operators on a weak
$V$-module $M$ for a set of generators which are mutually local (see \cite{Li2}).  These $g$-twisted vertex operators generate a local system which is a vertex superalgebra.  We then construct a homomorphism of vertex 
superalgebras {}from $V^{\otimes k}$ to this local system which thus gives a
weak $g$-twisted $V^{\otimes k}$-module structure on $M$.

For $v\in V$, and $k$ any positive integer, denote by $v^j\in\vtk$, for $j= 1, \dots, k$, the vector whose $j$-th tensor factor is $v$ and whose other tensor factors are ${\bf 1}$.  Then
$gv^j=v^{j-1}$ for $j=1,\dots,k$ where $0$ is understood to be $k$. 

Suppose that $W$ is a weak $g$-twisted $\vtk$-module, and let $\eta$ be a fixed primitive $k$-the root of unity.  We first make some general observations for this setting.  First, it follows {}from the definition of twisted module (cf. (\ref{limit-axiom})) that the $g$-twisted vertex operators on $W$ satisfy
\begin{equation}
Y_g(v^{j+1},x) =  Y_g(g^{-j} v^1,x) = \lim_{x^{1/k}\to \eta^{j}x^{1/k}}  Y_g(v^1,x).
\end{equation}
Since $\vtk$ is generated by $v^j$ for $v\in V$ and $j=1,...,k,$ 
the twisted vertex operators $Y_g(v^1,x)$ for $v\in V$ determine all
the twisted vertex operators $Y_g(u,x)$ on $W$ for any $u\in\vtk$. This
observation is very important in our construction of twisted modules.

Secondly, if $u,v\in V$ are of homogeneous sign, then by (\ref{twisted-Jacobi}) the twisted Jacobi identity for $Y_g(u^1,x_1)$ and $Y_g(v^1,x_2)$ is
\begin{multline}\label{k1}
x^{-1}_0\delta\left(\frac{x_1-x_2}{x_0}\right)
Y_g(u^1,x_1)Y_g(v^1,x_2)\\
-(-1)^{ |u||v|} x^{-1}_0\delta\left(\frac{x_2-x_1}{-x_0}\right)
Y_g(v^1,x_2)Y_g(u^1,x_1)\\
=\frac{1}{k}x_2^{-1}\sum_{j=0}^{k-1}\delta\Biggl(\eta^j\frac{(x_1-x_0)^
{1/k}}{x_2^{1/k}}\Biggr)Y_g(Y(g^ju^1,x_0)v^1,x_2).
\end{multline} 
Since $g^{-j}u^1=u^{j+1}$, we see that $Y(g^{-j}u^1,x_0)v^1$ only involves nonnegative integer powers of $x_0$ unless $j=0\ (\mod \; k).$   Thus the we have the supercommutator
\begin{multline}\label{k2}
[Y_g(u^1,x_1),Y_g(v^1,x_2)] \\
= \; \Res_{x_0}\frac{1}{k}x_2^{-1}\delta
\Biggl(\frac{(x_1-x_0)^{1/k}}{x_2^{1/k}}\Biggr)Y_g(Y(u^1,x_0)v^1,x_2) .
\end{multline}
This shows that the component operators of $Y_g(u^1,x)$ for $u\in V$ on
$W$ form a Lie superalgebra.

For $u\in V$ and $\D(z) = 
\D^V(x)$ given by (\ref{Delta-for-a-module}), define
\begin{equation}
\bar Y(u,x)=Y_M(\D(x)u,x^{1/k}) .
\end{equation}
For example, as in \cite{BDM}, taking $u=\omega$, and recalling that $a_2= (k^2-1)/12$, we have
\begin{eqnarray}\label{sun1}
\bY(\omega,x) &=&Y_M \left(\frac{x^{2(1/k - 1)}}{k^2}\Bigl(\omega + a_2
\frac{c}{2}x^{-2/k} \Bigr),x^{1/k}\right) \\
&=&  \frac{x^{2(1/k - 1)}}{k^2}Y(\omega,x^{1/k})
+\frac{(k^2-1)c}{24k^2}x^{-2} \nonumber
\end{eqnarray}
where $c$ is the central charge of $V$.

\begin{rem}\label{wrong-space-remark}{\em
We have
\begin{equation}\label{where-operators-live}
Y_M(\D(x)u,x^{1/k}) \in \left\{ \begin{array}{ll}
 (\mathrm{End}\, M) [[x^{1/k}, x^{-1/k}]] & \mbox{if $k$ is odd}\\
\\
 x^{|v|/2k} (\mathrm{End}\, M) [[x^{1/k}, x^{-1/k}]] & \mbox{ if $k$ is even}
\end{array}
\right. .
\end{equation}
When we put a weak $g$-twisted $\vtk$-module structure on $M$, this operator 
$\bar{Y}(u,x) = Y_M(\D(x)u,x^{1/k})$ will be the twisted vertex operator acting on $M$
associated to $u^1$, when $k$ is odd.  However, since this operator contains powers of $x^{1/2k}$ for $k$ even, it cannot be the twisted vertex operator associated to $u^1$ in this case.  Rather, as noted in Remark \ref{generalized-remark} below, they are a type of ``generalized" twisted vertex operator, generalizing the notion of ``relativized" twisted vertex operator for lattice vertex operator superalgebras as constructed in \cite{DL2}.}
\end{rem}

We next study the properties of the operators $\bY(u,x)$, following and generalizing the results of \cite{BDM}.

\begin{lem}\label{l3.1} For $u\in V$
\[\bY(L(-1)u,x)=\frac{d}{dx}\bY(u,x).\]
\end{lem}

\begin{proof} By Corollary \ref{c2.5}, the analogous proof given in \cite{BDM} (Lemma 3.2 of \cite{BDM}) in the nonsuper setting follows in the setting of vertex operator superalgebras.  That is, we have
\begin{eqnarray*}
\bY(L(-1)u,x) &=& Y_M(\D(x)L(-1)u,x^{1/k})\\
&=& Y_M(\frac{d}{dx}\D(x)u,x^{1/k}) + k^{-1} x^{1/k - 1}Y_M(L(-1)\D(x)u,x^{1/k})\\
&=& Y_M(\frac{d}{dx}\D(z)u,x^{1/k}) + \left. k^{-1} x^{1/k - 1}\frac{d}{dy}
Y_M(\D(x)u,y) \right|_{y=x^{1/k}}\\
\end{eqnarray*}
\begin{eqnarray*}
&=& Y_M(\frac{d}{dx}\D(x)u,x^{1/k}) + \left.\frac{d}{dy}Y_M(\D(x)u,y^{1/k})
\right|_{y=x}\\
&=& \frac{d}{dx}Y_M(\D(x)u,x^{1/k})\\
&=& \frac{d}{dx}\bY(u,x)
\end{eqnarray*}
as desired. \end{proof}

\begin{lem}\label{l3.2} For $u,v\in V$ of homogeneous sign, we have the supercommutator
\begin{multline}\label{first-supercommutator}
 [\bY(u,x_1),\bY(v,x_2)] \\
= \; \Res_{x_0} \frac{x_2^{-1}}{ k} 
\delta\Biggl(\frac{(x_1-x_0)^{1/k}}{x_2^{1/k}}\Biggr)\bY(Y(u,x_0)v,x_2) \left( \frac{x_1 - x_0}{x_2} \right)^{ \frac{|u|}{2k}(1-k)}, 
\end{multline}
or equivalently
\begin{multline}\label{second-supercommutator}
 [\bY(u,x_1),\bY(v,x_2)] \\
= \left\{ \begin{array}{ll}
\Res_{x_0}\frac{x_2^{-1}}{k}
\delta\Biggl(\frac{(x_1-x_0)^{1/k}}{x_2^{1/k}}\Biggr)\bY(Y(u,x_0)v,x_2) & \mbox{if $k$ is odd}\\
\Res_{x_0} \frac{x_2^{-1}}{k}
\delta\Biggl(\frac{(x_1-x_0)^{1/k}}{x_2^{1/k}}\Biggr)\bY(Y(u,x_0)v,x_2) \left( \frac{x_1 - x_0}{x_2} \right)^{ \frac{|u|}{2k}} & \mbox{if $k$ is even}
\end{array} . \right.
\end{multline}
\end{lem}

\begin{proof} The supercommutator formula for the weak $V$-module $M$ is given by 
\begin{equation}\label{commutator for Lemma 3.3}
[Y_M(u,x_1),Y_M(v,x_2)] \; = \; \Res_{x}x_2^{-1}\delta\left(\frac{x_1-x}{x_2}
\right)Y_M(Y(u,x)v,x_2)
\end{equation}
which is a consequence of the Jacobi identity on $M$.  Replacing $Y_M(u,x_1)$ and $Y_M(v,x_2)$ by $Y_M(\D(x_1)u,x_1^{1/k})$ and 
$Y_M(\D(x_2)v,x_2^{1/k})$, respectively, in the supercommutator
formula, we have the supercommutator
\begin{multline}\label{substitution equation}
[\bY(u,x_1),\bY(v,x_2)] \\
=  \Res_{x}x_2^{-1/k}\delta\Biggl(
\frac{x_1^{1/k}-x}{x_2^{1/k}}\Biggr)
Y_M(Y(\D(x_1)u,x)\D(x_2)v,x_2^{1/k}).
\end{multline}

We want to make the change of variable $x = x_1^{1/k}-(x_1-x_0)^{1/k}$
where by $x_1^{1/k}-(x_1-x_0)^{1/k}$ we mean the power series expansion
in positive powers of $x_0$.  For $n \in \mathbb{Z}$, it was shown in \cite{BDM} that
\begin{equation}
\left. (x_1^{1/k} - x)^n \right|_{x = x_1^{1/k}-(x_1-x_0)^{1/k}} = (x_1 - x_0)^{n/k}.
\end{equation}

Thus substituting $x =  x_1^{1/k}-(x_1-x_0)^{1/k}$ into 
\[x_2^{-1/k}\delta\Biggl(\frac{x_1^{1/k}-x}{x_2^{1/k}}\Biggr)
Y_M(Y(\D(x_1)u,x)\D(x_2)v,x_2^{1/k}) \]
we have a well-defined power series given by
\[x_2^{-1/k} \delta\Biggl(\frac{(x_1-x_0)^{1/k}}{x_2^{1/k}}\Biggr)
 Y_M(Y(\D(x_1)u, x_1^{1/k}-(x_1-x_0)^{1/k})\D(x_2)v,x_2^{1/k}) .\] 

Let $f(z_1,z_2,x)$ be a complex analytic function in $z_1, z_2$, and
$x$, and let $h(z_1,z_2,z_0)$ be a complex analytic function in $z_1,
z_2$, and $z_0$. Then if $f(z_1,z_2,h(z_1,z_2,z_0))$ is well defined,
and thinking of $z_1$ and $z_2$ as fixed, i.e., considering 
$f(z_1,z_2,h(z_1,z_2,z_0))$ as a Laurent series in $z_0$, by the
residue theorem of complex analysis, we have
\begin{equation}\label{residue change of variables}
\Res_x f(z_1,z_2,x)=\Res_{z_0} \left( \frac{\partial}{\partial z_0}
h(z_1,z_2,z_0) \right)f(z_1,z_2,h(z_1,z_2,z_0))
\end{equation} 
which of course remains true for $f$ and $h$ formal power series in
their respective variables.  Thus making the change of variable
$x= h(x_1,x_2,x_0) = x_1^{1/k}-(x_1-x_0)^{1/k}$, using (\ref{substitution
equation}), (\ref{residue change of variables}), the 
$\delta$-function identity (\ref{delta-function3}), (\ref{where-operators-live}), and Proposition 
\ref{psun1}, we obtain
\begin{eqnarray*}
\lefteqn{ [\bY(u,x_1),\bY(v,x_2)] = }\\  
&=& \Res_{x_0}\frac{1}{k}x_2^{-1/k} (x_1-x_0)^{1/k-1}\delta\Biggl(
\frac{(x_1-x_0)^{1/k}}{x_2^{1/k}}\Biggr) \\
& & \quad Y_M(Y(\D(x_1)u,x_1^{1/k}-(x_1-x_0)^{1/k})
\D(x_2)v,x_2^{1/k})\\
&=& \Res_{x_0}\frac{1}{k} x_2^{-1} \delta\Biggl(
\frac{(x_1-x_0)^{1/k}}{x_2^{1/k}}\Biggr)  \\
& & \quad 
Y_M(Y(\D(x_1)u,x_1^{1/k}-(x_1-x_0)^{1/k})\D(x_2)v,x_2^{1/k})\\
&=&  \Res_{x_0}\frac{1}{k}x_1^{-1} \delta\Biggl(
\frac{(x_2+x_0)^{1/k}}{x_1^{1/k}}\Biggr) \\
& & \quad 
Y_M(Y(\D(x_1)u,x_1^{1/k}-(x_1-x_0)^{1/k})\D(x_2)v,x_2^{1/k})
\end{eqnarray*}
Now we observe that  
\begin{multline}
Y_M (Y(\D(x_1)u,x_1^{1/k}-(x_1-x_0)^{1/k})\D(x_2)v,x_2^{1/k}) \\
\in 
\left\{ \begin{array}{ll}
x^{|u|/2k}_1 x_2^{|v|/2k} (\mathrm{End} \, M)  [[x_0]] [[x_1^{1/k}, x_1^{-1/k}]][[x_2^{1/k}, x_2^{-1/k}]] & \mbox{if $k$ is even}\\
\\
(\mathrm{End} \, M)  [[x_0]] [[x_1^{1/k}, x_1^{-1/k}]][[x_2^{1/k}, x_2^{-1/k}]] & \mbox{if $k$ is odd} 
\end{array}
\right. .
\end{multline}
Thus using the $\delta$-function substitution property and Proposition \ref{psun1}, we obtain
\begin{eqnarray*}
\lefteqn{ [\bY(u,x_1),\bY(v,x_2)] = }\\
&=& \Res_{x_0}\frac{1}{k}x_1^{-1}\delta\Biggl(
\frac{(x_2+x_0)^{1/k}}{x_1^{1/k}}\Biggr)   \left( \frac{x_2 + x_0}{x_1} \right)^{\frac{|u|}{2k}(k-1)} \\
& & \quad  Y_M(Y(\D(x_2+x_0)u,(x_2+x_0)^{1/k}-x_2^{1/k})
\D(x_2)v,x_2^{1/k})\\
&=& \Res_{x_0}\frac{1}{k}x_2^{-1}\delta\Biggl(
\frac{(x_1-x_0)^{1/k}}{x_2^{1/k}}\Biggr)   \left( \frac{x_1 - x_0}{x_2} \right)^{-\frac{|u|}{2k}(k-1)} \\
& & \quad Y_M(Y(\D(x_2+x_0)u,(x_2+x_0)^{1/k}-x_2^{1/k}) 
\D(x_2)v,x_2^{1/k}) \\
&=&  \Res_{x_0}\frac{1}{k}x_2^{-1}\delta\Biggl(\frac{(x_1-x_0)^
{1/k}}{x_2^{1/k}}\Biggr)Y_M(\D(x_2)Y(u,x_0)v,x_2^{1/k})  \left( \frac{x_1 - x_0}{x_2} \right)^{\frac{|u|}{2k}(1-k)} 
\end{eqnarray*}
\begin{eqnarray*}
&=& \Res_{x_0}\frac{1}{k}x_2^{-1}\delta\Biggl(\frac{(x_1-x_0)^
{1/k}}{x_2^{1/k}}\Biggr)\bY(Y(u,x_0)v,x_2)  \left( \frac{x_1 - x_0}{x_2} \right)^{\frac{|u|}{2k}(1-k)} , \hspace{.3in}
\end{eqnarray*}
giving (\ref{first-supercommutator}).  Equation (\ref{second-supercommutator}) follows from the properties of the $\delta$-function. \end{proof}

\section{The construction of a weak $(1 \; 2 \; \cdots \; k$)-twisted $V^{\otimes k}$-module structure on a weak $V$-module $(M, Y_M)$ for $k$ odd}\label{tensor-product-twisted-construction-section}
\setcounter{equation}{0}

Let $M=(M,Y_M)$ be a weak $V$-module. Now we begin our construction of a weak $g$-twisted $V^{\otimes k}$-module structure on $M$ when $k$ is an odd positive integer and $g = (1 \; 2\; \cdots \; k)$.  Since establishing the properties of $\Delta_k(x)$ as in Section \ref{Delta-section} in the super case and proving the supercommutator (\ref{first-supercommutator}), our construction of a weak $g$-twisted $V^{\otimes k}$-module structure on $M$ in the case when $k$ is odd follows the same spirit as the construction in the nonsuper case given in \cite{BDM}, but now using the full power of local systems for twisted operators in the super case as established by Li in \cite{Li2}.

For $u\in V$ set
\begin{equation}  \label{define-g-twist} 
Y_g(u^1,x) = \bY(u,x)\quad \mbox{and}   \quad
Y_g(u^{j+1},x) = \lim_{x^{1/k}\to \eta^{j} x^{1/k}}  Y_g(u^1,x)  . 
\end{equation}

\begin{rem}\label{generalized-remark}{\em
From the supercommutator (\ref{first-supercommutator}) for $\bar{Y}$, we see that defining $g$-twisted operators as above for the case when $k$ is even, can not result in a twisted module structure on $M$ due to appearance of the extra term involving $(x_2^{-1} (x_1 - x_0))^{|u|/2k}$.  In particular, the most we could hope for would be a type of ``generalized" $g$-twisted $V^{\otimes k}$-module structure in the spirit of \cite{DL1} and the ``relativized" twisted vertex operators for lattice vertex operator superalgebras as constructed in \cite{DL2}.  
}
\end{rem}

Note that $Y_g(u^j,x)=\sum_{p=0}^{k-1}Y_g^p(u^j,x)$
where $Y_g^p(u^j,x)=\sum_{n\in \frac{p}{k} + \Z}u^j_nx^{-n-1}$.

\begin{lem}\label{l3.3} Let $u,v\in V$ of homogeneous sign.  Then we have the supercommutator
\begin{multline}\label{3.1}
[Y_g(u^j,x_1),Y_g(v^m,x_2)] \\
= \; \Res_{x_0}\frac{1}{k}x_2^{-1}\delta
\Biggl(\frac{\eta^{j-m}(x_1-x_0)^{1/k}}{x_2^{1/k}}\Biggr)Y_g((Y(u,x_0)v)^m,x_2) 
\end{multline}
where $(Y(u,x_0)v)^m=\sum_{n\in\Z}(u_nv)^m x_0^{-n-1},$
and
\begin{multline}\label{3.2}
[Y_g^p(u^j,x_1),Y_g(v^m,x_2)] \\
= \Res_{x_0}\frac{1}{k}x_2^{-1}\eta^{(m-j)p}\left(\frac{x_1-x_0}
{x_2}\right)^{-p/k}\delta\left(\frac{x_1-x_0}
{x_2}\right)Y_g((Y(u,x_0)v)^m,x_2).
\end{multline}
\end{lem}

\begin{proof} By Lemma \ref{l3.2}, equation (\ref{3.1}) holds if $j=m =1$ and $k$ odd.
Then using (\ref{define-g-twist}), we obtain equation (\ref{3.1})
for any $j,m = 1,...,k$. Equation (\ref{3.2}) is a direct consequence of
(\ref{3.1}). \end{proof}

%NOTE:  THERE IS A TYPO HERE IN BDM!!!! In Equation (3.15), the exponent of $\eta$ should
%be $(i-j)p$ not $(j-i)p$.

By Lemma \ref{l3.3} for $u,v\in V$ of homogeneous sign, there exists a positive integer
$N$ such that 
\begin{equation}\label{a3.3}
[Y_g(u^j,x_1), Y_g(v^m,x_2)](x_1-x_2)^N=0.
\end{equation}
Taking the limit $x^{1/k} \longrightarrow \eta^{j-1} x^{1/k}$
in Lemma \ref{l3.1}, for $j=1,\dots,k$, we have
\begin{equation}
Y_g(L(-1)u^j,x)=\frac{d}{dx}Y_g(u^j,x).
\end{equation} 
Thus the operators $Y_g(u^j,x)$ for $u \in V$, and for $j=1,\dots,k$ are mutually local and generate a local system $A$ of weak twisted vertex operators on $(M, L(-1))$ in the sense of \cite{Li2}. 

Let $\rho$ be a map {}from $A$ to $A$ such that $\rho Y_g(u^j,x)=Y_g(u^{j-1},x)$ for
$u\in V$ and $j=1,...,k$.  By Theorem 3.14 of \cite{Li2}\footnote{There is a
typo in the statement of Theorem 3.14 in \cite{Li2}.  The $V$ in the theorem
should be $A$.  That is, the main result of the theorem is that the
local system $A$ of the theorem has the structure of a vertex
superalgebra.}, the local system $A$ generates a vertex superalgebra we
denote by $(A,Y_A)$, and $\rho$ extends to an automorphism of $A$ of
order $k$ such that $M$ is a natural weak generalized $\rho$-twisted $A$-module
in the sense that $Y_A(\alpha(x),x_1)=\alpha(x_1)$ for $\alpha(x)\in A$
are $\rho$-twisted vertex operators on $M$.

\begin{rem} {\em $\rho$ is given by 
\begin{equation}
\rho a(x) =  \lim_{x^{1/k} \to \eta^{-1}x^{1/k}}  a(x) 
\end{equation}
for $a(x)\in A$; see \cite{Li2}.}
\end{rem}

Let $A^j=\{c(x)\in A| \rho c(x)=\eta^j c(x)\}$ and $a(x)\in A^j$ of homogeneous sign in $A$.
For any integer $n$ and $b(x)\in A$ of homogeneous sign, the operator $a(x)_{n}b(x)$ is an
element of $A$ given by
\begin{eqnarray}\label{3.3}
a(x)_{n}b(x)={\rm Res}_{x_{1}}{\rm
Res}_{x_{0}}\left(\frac{x_{1}-x_{0}}{x}\right)^{j/k}x_{0}^{n}\cdot X
\end{eqnarray}
where 
\[ X=x_{0}^{-1}\delta\left(\frac{x_{1}-x}{x_{0}}\right)a(x_{1})
b(x)- (-1)^{|a||b|} x_{0}^{-1}
\delta\left(\frac{x-x_{1}}{-x_{0}}\right)b(x)a(x_{1}).\]
Or, equivalently, $a(x)_{n}b(x)$ is defined by:
\begin{eqnarray}\label{3.4}
\sum_{n\in \mathbb{Z}}\left(a(x)_{n}b(x)\right)x_{0}^{-n-1}
={\rm Res}_{x_{1}}\left(\frac{x_{1}-x_{0}}{x}\right)^{j/k}
\cdot X.
\end{eqnarray}
Thus following \cite{Li2}, for $a(z)\in A^j$, we define $Y_A(a(z),x)$ 
by setting $Y_A(a(z), x_0)b(z)$ equal to (\ref{3.4}).
  
\begin{lem}\label{l3.5} For $u, v \in V$ of homogeneous sign, we have the supercommutator
\[ [Y_A(Y_g(u^j,x),x_1), Y_A(Y_g(v^m,x),x_2)]=0\]
for $j,m = 1,\dots, k$, with $j \neq m$.
\end{lem}

\begin{proof} The proof is analogous to the proof of Lemma 3.6 in \cite{BDM} where we use the vertex superalgebra structure of $A$ rather than just the vertex algebra structure and we use the supercommutators of Lemma \ref{l3.3}. \end{proof}

Let $Y_g(u^i,z)_n$ for $n \in \mathbb{Z}$ denote the coefficient of $x^{-n-1}$ in the vertex operator $Y_A(Y_g(u^i,z), x)$ for $u \in V$.  That is
\[Y_A(Y_g(u^i,z), x) = \sum_{n \in \Z} Y_g(u^i,z)_n \; x^{-n-1} \in (\mathrm{End} \; A) [[x, x^{-1}]].\]

\begin{lem}\label{l3.6} For $u_1,...,u_k\in V$, we have  
\begin{multline*}
Y_A(Y_g(u^1_1,z)_{-1}\cdots Y_g(u_{k-1}^{k-1},z)_{-1}Y_g(u_k^k,z),x)  \\
 = Y_A(Y_g(u^1_1,z),x)\cdots Y_A(Y_g(u_{k-1}^{k-1},z),x)Y_A(Y_g(u_k^k,z),x) .
\end{multline*}
\end{lem}

\begin{proof} From the Jacobi identity on $A$, and Lemma \ref{l3.5}, we have, for $1\leq i<j\leq k$,
\begin{eqnarray*}
\lefteqn{Y_A(Y_g(u^i, z)_{-1} Y_g(v^j,z), x)}\\
&=& \mathrm{Res}_{x_1} \mathrm{Res}_{x_0} x_0^{-1} \left( x^{-1}_0\delta\left(\frac{x_1-x}{x_0}\right) Y_A(Y_g(u^i,z),x_1)Y_A(Y_g(v^j,z),x) \right. \\
& & \quad \left. - (-1)^{|u||v|}
x^{-1}_0\delta\left(\frac{x-x_1}{-x_0}\right) Y_A(Y_g(v^j,z),x)Y_A(Y_g(u^i,z),x_1)\right)\\
&=& \mathrm{Res}_{x_1} \Bigl(  (x_1-x)^{-1} Y_A(Y_g(u^i,z),x_1)Y_A(Y_g(v^j,z),x)  \\
& & \quad - (-1)^{|u||v|} (x-x_1)^{-1} Y_A(Y_g(v^j,z),x)Y_A(Y_g(u^i,z),x_1)\Bigr)\\
&=& \sum_{n <0} Y_g(u^i,z)_n x^{-n-1} Y_A(Y_g(v^j,z),x) \\
& & \quad - (-1)^{|u||v|}  Y_A(Y_g(v^j,z),x) \sum_{n \geq 0} Y_g(u^i,z)_n x^{-n-1}  \\
&=&  \sum_{n \in \Z} Y_g(u^i,z)_n x^{-n-1} Y_A(Y_g(v^j,z),x) \\
&=& Y_A(Y_g(u^i, z), x) Y_A(Y_g(v^j,z), x).
\end{eqnarray*}
The result follows by induction.  
\end{proof}

Define the map $f : \vtk \longrightarrow A$ by
\begin{eqnarray*}
f: \vtk &\longrightarrow& A\\
u_1\otimes\cdots \otimes u_k = (u_1^1)_{-1}\cdots (u_{k-1}^{k-1})_{-1}u^k_k \! \! \! &\mapsto& \! \! \! 
Y_g(u_1^1,z)_{-1}\cdots 
Y_g(u_{k-1}^{k-1},z)_{-1}Y_g(u_k^k,z) 
\end{eqnarray*}
for $u_1,...,u_k\in V$. Then $f(u^j)=Y_g(u^j,z).$

\begin{lem}\label{l3.7} $f$ is a homomorphism of vertex superalgebras.
\end{lem}

\begin{proof} We need to show that
\[fY(u_1\otimes\cdots \otimes u_k,x)=Y_A( Y_g(u_1^1,z)_{-1}\cdots 
Y_g(u_{k-1}^{k-1},z)_{-1}Y_g(u_k^k,z),x)f\]
for $u_i\in V.$ 
Take $v_i\in V$ for $i=1,...,k.$ Then 
\begin{eqnarray*}
\lefteqn{fY(u_1\otimes\cdots \otimes u_k,x)(v_1\otimes \cdots\otimes v_k) }\\
&=& \! \! (-1)^s f(Y(u_1,x)v_1\otimes\cdots Y(u_k,x)v_k)\\
&=& \! \! (-1)^s Y_g(Y(u_1^1,x)v_1^1,z)_{-1}\cdots Y_g(Y(u_{k-1}^{k-1},x)v_{k-1}^{k-1},z)_{-1}
Y_g(Y(u_k^k,x)v_k^k,z) 
\end{eqnarray*}
for $s = \sum_{j=1}^{k-1} |v_j| \sum_{i = j + 1}^k |u_i|$.

By Lemma \ref{l3.6}, we have
\begin{multline*} 
Y_A( Y_g(u_1^1,z)_{-1}\cdots Y_g(u_{k-1}^{k-1},z)_{-1}Y_g(u_k^k,z),x)
f(v^1\otimes \cdots \otimes v^k) \\
= Y_A(Y_g(u^1_1,z),x)\cdots Y_A(Y_g(u_{k-1}^{k-1},z),x)Y_A(Y_g(u_k^k,z),x)
Y_g(v_1^1,z)_{-1}\\
\cdots Y_g(v_{k-1}^{k-1},z)_{-1}Y_g(v_k^k,z).
\end{multline*}
By Lemma \ref{l3.5}, it is enough to show that 
$$ Y_g(Y(u^j,x)v^j,z)=Y_A(Y_g(u^,z),x)Y_g(v^j,z)$$
for $u,v\in V$ and $j=1,...,k.$ In fact, in view of the
relation between $Y(u^1,z)$ and $Y(u^j,z)$ for 
$u\in V,$  we only need to prove the case $j=1.$

By Proposition \ref{psun1},
\begin{eqnarray*}
Y_g(Y(u^1,x_0)v^1,x_2) &=& Y_M(\Delta_k(x_2)Y(u,x_0)v,x_2^{1/k})\\
&=& Y_M(Y(\Delta_k(x_2+x_0)u,(x_2+x_0)^{1/k}-x_2^{1/k})\Delta_k(x_2)v,x_2^{1/k}).
\end{eqnarray*}
On the other hand,
\[Y_A(Y_g(u^1,x_2),x_0)Y_g(v^1,x_2)=\sum_{p=0}^{k-1}\Res_{x_1}
\left(\frac{x_1-x_{0}}{x_2}\right)^{p/k}X\]
where
\begin{multline}
X=x_{0}^{-1}\delta\left(\frac{x_1-x_2}{x_{0}}\right)Y_g(u^1,x_1)
Y_g(v^1,x_2)\\
-(-1)^{|u||v|}x_{0}^{-1}\delta\left(\frac{x_2-x_1}{-x_{0}}\right)
Y_g(v^1,x_2)Y_g(u^1,x_1).
\end{multline}

By equation (\ref{a3.3}), there exists a positive integer $N$ such that
\[(x_1-x_2)^NY_g(u^1,x_1)Y_g(v^1,x_2)= (-1)^{|u||v|} (x_1-x_2)^NY_g(v^1,x_2)Y_g(u^1,x_1).\]
Thus, using the three-term $\delta$-function identity (\ref{three-term-delta}), we have 
\begin{eqnarray*}
X &=& x_{0}^{-1}\delta\left(\frac{x_1-x_2}{x_{0}}\right)Y_g(u^1,x_1)
Y_g(v^1,x_2)\\
& &\quad - (-1)^{|u||v|} x_{0}^{-1} \delta\left(\frac{x_2-x_1}{-x_{0}}\right)
x_0^{-N}(x_1-x_2)^NY_g(v^1,x_2)Y_g(u^1,x_1)\\
&=& x_{0}^{-1}\delta\left(\frac{x_1-x_2}{x_{0}}\right)x_0^{-N}
\left((x_1-x_2)^NY_g(u^1,x_1)Y_g(v^1,x_2)\right)\\
& & \quad - (-1)^{|u||v|} x_{0}^{-1} \delta\left(\frac{x_2-x_1}{-x_{0}}\right)
x_0^{-N}\left((x_1-x_2)^NY_g(u^1,x_1)Y_g(v^1,x_2)\right)\\
&=& x_2^{-1}x_0^{-N}\delta\left(\frac{x_1-x_0}{x_2}\right)
\left((x_1-x_2)^NY_g(u^1,x_1)Y_g(v^1,x_2)\right) .
\end{eqnarray*}

Therefore using the $\delta$-function relation (\ref{delta-function2}), 
we have
\begin{multline*}
Y_A(Y_g(u^1,x_2),x_0)Y_g(v^1,x_2) \\
= \Res_{x_1} x_0^{-N} x_2^{-1} \delta\Biggl(\frac{(x_1-x_0)^{1/k}}{x_2^{1/k}}\Biggr) \left((x_1-x_2)^NY_g(u^1,x_1)Y_g(v^1,x_2)\right).
\end{multline*}

And the rest of the proof is analogous to the corresponding part of the proof of Lemma 3.8 in \cite{BDM}. \end{proof}

Let $(M,Y)$ be a weak $V$-module, $k$ a positive odd integer, and $g = (1 \; 2 \; \cdots \; k)$.  Define $T_g^k(M,Y) = (T_g^k(M),
Y_g) = (M, Y_g)$.  That is $T_g^k(M, Y)$ is $M$ as the underlying
vector space and the vertex operator $Y_g$ is given by (\ref{define-g-twist}).

Now we state our first main theorem of the paper.
\begin{thm}\label{main1} 
$(T_g^k(M),Y_g)$ is a weak $g$-twisted $V^{\otimes k}$-module such
that $T_g^k(M)=M$, and $Y_g$, defined by (\ref{define-g-twist}), is the
linear map {}from $V^{\otimes k}$ to $(\End \; T_g^k(M))[[x^{1/k},\\
x^{-1/k}]]$
defining the twisted module structure. Moreover, 

(1) $(M, Y)$ is an irreducible weak $V$-module if and only if
$(T_g^k(M), Y_g)$ is an irreducible weak $g$-twisted $V^{\otimes
k}$-module.

(2) $M$ is a weak admissible $V$-module if and only if $T_g^k(M)$ is a weak 
 admissible $g$-twisted $V^{\otimes k}$-module.

(3) $M$ is an ordinary $V$-module if and only if $T_g^k(M)$ is an
 ordinary $g$-twisted $V^{\otimes k}$-module.
\end{thm}

\begin{proof} It is immediate {}from Lemma \ref{l3.7} that $T_g^k(M)=M$ is a weak $g$-twisted $V^{\otimes k}$-module with $Y_g(u^1,x)=\bar Y(u,x).$ Note
that with 
\begin{equation}\label{Delta-inverse}
\Delta_k (x)^{-1} = ( x^{1/2k})^{-( 1 - k) 2L(0)}  (k^{1/2})^{2L(0)} \exp 
\Biggl( -\sum_{j \in \ZZ} a_j x^{- j/k} L(j) \Biggr),
\end{equation}
we have 
\begin{equation}\label{for-grading}
Y_g((\Delta_k(x^k)^{-1}u)^1,x)=\bar Y(\Delta_k(x)^{-1}u,x)=Y_M(u,x^{1/k}),
\end{equation}
and all twisted vertex operators $Y_g(v,x)$ for $v\in \vtk$ can
be generated {}from $Y_g(u^1,x)$ for $u\in V.$  It is clear now that $M$
is an irreducible weak $V$-module if and only if $T_g^k(M)$ is an
irreducible weak $g$-twisted $V^{\otimes k}$-module, proving statement (1).

For statement (2), we first assume that $M$ is a weak admissible $V$-module, i.e. we have $M =
\coprod_{n\in \frac{1}{2}\mathbb{N}} M(n)$ such that for $m \in 
\mathbb{Z}$, the component operator $u_m$ of $Y_M(u, z)$, satisfies $u_mM(n)\subset
M(\wt \; u-m-1+n)$ if $u\in V$ is of homogeneous weight. Define a
$\frac{1}{2k}\N$-gradation on $T_g^k(M)$ such that $T_g^k(M)(n/k) =
M(n)$ for $n\in \frac{1}{2}\Z.$  Recall that $Y_g(v,x) = \sum_{m \in \frac{1}{k}
\mathbb{Z}} v^g_m x^{-m-1}$ for $v\in \vtk$.  We need to show that
$v^g_mT_g^k(M)(n)\subset T_g^k(M)(\wt \; v-m-1+n)$ for $m\in\frac{1}{k} \Z$, and $n \in \frac{1}{2k} \N$.  Since all twisted vertex operators $Y_g(v,x)$ for
$v\in \vtk$ can be generated {}from $Y_g(u^1,x)$ for $u\in V$, it is
enough to show $(u^1)^g_m T_g^k(M)(n) \subset T_g^k(M)(\wt \; u-m-1+n)$.

Let $u\in V_p$ for $p\in\frac{1}{2} \Z$.  Then 
\[\Delta_k(x)u=\sum_{j=0}^{\infty}u(j)x^{p/k - p - j/k}\]
where $u(j)\in V_{p-j}.$ Thus
\[Y_g(u^1,x)=Y_M(\Delta_k(x)u,x^{1/k})=\sum_{j=0}^{\infty}Y_M(u(j),x^{1/k})
x^{p/k - p - j/k} ,\]
and thus for $m \in \frac{1}{k} \mathbb{Z}$
\[(u^1)^g_m=\sum_{j=0}^{\infty}u(j)_{(1-k)p-j-1+km+k}.\]
Since the weight of $u(j)_{(1-k)p-j-1+km+k}$ is $k(p-m-1)$, we see
that for $n \in \frac{1}{2k} \N$, that $(u^1)^g_mT_g^k(M)(n)=(u^1)^g_mM(kn) \subset M(k(p-m-1+n)) =T_g^k(M)(p-m-1+n)$, showing that $T_g^k(M)$
is a weak admissible $g$-twisted $V^{\otimes k}$-module.

Conversely, we assume that $T_g^k(M)$ is a weak admissible $g$-twisted $V^{\otimes k}$-module, i.e. we have $T_g^k(M) = \coprod_{n\in \frac{1}{2k}\mathbb{N}} T_g^k(M)(n)$ such that for $m \in \frac{1}{k} \mathbb{Z}$, the component operator $u^g_m$ of $Y_g(u, x)$ satisfies $u^g_m T_g^k(M)(n)\subset T_g^k(M)(\wt \; u-m-1+n)$ if $u\in V^{\otimes k}$ is of homogeneous weight. Define a $\frac{1}{2}\N$-gradation on $M$ such that $M(n) = T_g^k(M)(n/k)$ for $n\in \frac{1}{2}\Z.$ 

Note that by again letting $u\in V_p$ for $p\in\frac{1}{2} \Z$, then 
\[\Delta_k(x)^{-1} u=\sum_{j=0}^{\infty}u[j]x^{ p -p/k - j}\]
where $u[j]\in V_{p-j}.$  Thus Equation (\ref{for-grading}) implies 
\[Y_M(u, x) = Y_g((\Delta_k(x^k)^{-1}u)^1,x^k) = \sum_{j=0}^\infty Y_g(u[j]^1, x^k) x^{pk - p -jk} \]
and thus for $m \in \mathbb{Z}$
\[u_m = \sum_{j=0}^\infty (u[j]^1)^g_{\frac{1}{k}((k-1)p -jk -k + m + 1) }.\]
The weight of $(u[j]^1)^g_{\frac{1}{k}((k-1)p -jk -k + m + 1) }$ is $\frac{1}{k}(p-m-1)$.  Therefore for the weak $V$-module $M$, we have $u_m M(n) = u_m T_g^k(M) (n/k) \subset T_g^k(M) ( \frac{1}{k} (p - m - 1 + n)) = M(p-m-1 + n)$, finishing the proof of (2).

In order to prove (3) we write $Y_g(\bar\omega,x) = \sum_{n\in\Z}
L^g(n)x^{-n-2}$ where $\bar\omega=\sum_{j=1}^k\omega^j$.  We have
\[Y_g(\bar\omega,x)=\sum_{j=0}^{k-1}\lim_{x^{1/k}\mapsto \eta^{-j}x^{1/k}}
  Y_g(\omega^1,x).\]
It follows {}from (\ref{sun1}) that $L^g(0)=\frac{1}{k}L(0)+\frac{(k^2-1)c}{24k}$, immediately implying (3).
\end{proof}

Let $V$ be an arbitrary vertex operator superalgebra and $g$ an
automorphism of $V$ of finite order. We denote the categories of weak,
weak admissible and ordinary generalized $g$-twisted $V$-modules by $\mathcal{ C}^g_w(V),$
$\mathcal{ C}^g_a(V)$ and $\mathcal{ C}^g(V)$, respectively.  If $g=1$, we
habitually remove the index $g.$

Now again consider the vertex operator superalgebra $V^{\otimes k}$ and the
$k$-cycle $g = (1 \; 2 \; \cdots \; k)$ for $k$ odd.  Define
\begin{eqnarray*}
T_g^k: \mathcal{ C}_w(V) &\longrightarrow& \mathcal{ C}^g_w(V^{\otimes k})\\
  (M,Y) &\mapsto& (T_g^k(M),Y_g) = (M,Y_g)\\
     f  &\mapsto& T_g^k(f) = f
\end{eqnarray*}
for $(M,Y_M)$ an object and $f$ a morphism in $\mathcal{ C}_w(V)$.  

The following corollary to Theorem \ref{main1} follows immediately.

\begin{cor}\label{c3.10} 
If $k$ is odd, then $T_g^k$ is a functor {}from the category $\mathcal{ C}_w(V)$ to the category
$\mathcal{ C}^g_w(\vtk)$ such that: (1) $T_g^k$ preserves irreducible
objects; (2) The restrictions of $T_g^k$ to $\mathcal{ C}_a(V)$ and $\mathcal{
C}(V)$ are functors {}from $\mathcal{ C}_a(V)$ and $\mathcal{ C}(V)$ to $\mathcal{
C}^g_a(\vtk)$ and $\mathcal{ C}^g(\vtk)$, respectively.
\end{cor}

In the next section we will construct a functor $U_g^k$, in the case when $k$ is odd, {}from the
category $\mathcal{ C}^g_w(\vtk)$ to the category $\mathcal{ C}_w(V)$ such
that $U_g^k \circ T_g^k = id_{\mathcal{ C}_w(V)}$ and $T_g^k \circ U_g^k = 
id_{\mathcal{ C}^g_w(\vtk)}$.

\section{Constructing a weak $V$-module structure on a 
weak $g = (1 \; 2 \; \cdots \; k)$-twisted $V^{\otimes k}$-module for $k$ odd}\label{classification-section}
\setcounter{equation}{0}

For $k \in \ZZ$ and $g = (1\; 2\; \cdots \; k)$, let $M=(M,Y_g)$ be a 
weak $g$-twisted $\vtk$-module.  Motivated by the construction 
of weak $g$-twisted $\vtk$-modules {}from weak $V$-modules in 
Section 5, we consider 
\begin{equation}\label{define U}
Y_g((\Delta_k(x^k)^{-1}u)^1,x^k)
\end{equation}
for $u\in V$ where $\Delta_k (x)^{-1} = \Delta_k^V(x)^{-1}$ is given by 
(\ref{Delta-inverse}).
Note that (\ref{define U}) is multivalued since 
$Y_g((\Delta_k(x)^{-1}u)^1,x) \in (\mathrm{End} \, M) [[x^{1/2k}, x^{-1/2k}]]$.   
Thus we define 
\begin{equation}
Y_M(u,x) = Y_g((\Delta_k(x^k)^{-1}u)^1,x^k)
\end{equation} 
to be the unique formal Laurent series in $(\mathrm{End} \,  M) [[x^{1/2}, x^{-1/2}]]$ given by  
taking $(x^k)^{1/k} = x$.   Note that if $k$ is odd, then $Y_M(u,x) \in (\mathrm{End} \, M) [[x, x^{-1}]]$.

Our goal in this section is for the case when $k$ is odd, to construct a functor $U_g^k : \mathcal{ C}_w^g(\vtk) \rightarrow \mathcal{ C}_w(V)$ with $U_g^k(M_g,Y_g) = (U_g^k(M_g),Y_M) = (M_g,Y_M)$.  If we instead define $Y_M$ by taking $(x^k)^{1/k} = \eta^jx$ for $\eta$ a fixed primitive $k$-th root of unity for $j=1,\dots,k-1$, then $(M_g,Y_M)$ will not be a weak $V$-module.  Further note that this implies that if we allow $x=z$ to be complex number and if we define $z^{1/k}$ using the principal branch of the logarithm, then much of our work in this section is valid if and only if $-\pi/k < \mathrm{arg} \; z < \pi/k$.

\begin{lem}\label{l4.1} For $u\in V,$ we have
\begin{eqnarray*}
Y_M(L(-1)u,x) &=& \left(\frac{d}{dx}((x^k)^{1/k})\right)\frac{d}{dx}Y_M(u,x)\\
&=& \frac{d}{dx} Y_M(u,x)
\end{eqnarray*}
on $U_g^k(M_g) = M_g$.  Thus the $L(-1)$-derivative property holds for $Y_M$ on $M_g$.
\end{lem}

\begin{proof} The proof is similar to that of Lemma \ref{l3.1}, and follows from Corollary \ref{c2.5} via the analogous proof of Lemma 4.1 in \cite{BDM}. \end{proof}

\begin{lem}\label{l4.2} Let $u,v\in V$.  Then on $U_g^k(M_g) = M_g$, we have the supercommutator
\begin{multline}
[Y_M (u,x_1),Y_M (v,x_2)] \\
=
\Res_{x_0} x_2^{-1}\delta\left(\frac{x_1-x_0}{x_2}\right) \left(\frac{x_1-x_0}{x_2}\right)^{\frac{1}{2}(1-k)|u|} 
Y_M (Y(u, x_0)v,x_2),
\end{multline}
i.e.,
\begin{multline}
[Y_M (u,x_1), Y_M (v,x_2) ] \\
= \left\{ \begin{array}{ll}
\Res_{x_0} x_2^{-1}\delta\left(\frac{x_1-x_0}{x_2}\right) \left(\frac{x_1-x_0}{x_2}\right)^{\frac{|u|}{2}} Y_M (Y(u, x_0)v,x_2) & \mbox{if $k$ is even}\\
\Res_{x_0} x_2^{-1}\delta\left(\frac{x_1-x_0}{x_2}\right) Y_M (Y(u, x_0)v,x_2) & \mbox{if $k$ is odd}
\end{array} \right. .
\end{multline}
Therefore, the operators $Y_M$ can satisfy the Jacobi identity only if $k$ is odd. 
\end{lem}

\begin{proof} 
The proof is similar to the proof of Lemma \ref{l3.2} and is analogous to the proof of Lemma 4.2 in \cite{BDM}. 
From the 
twisted Jacobi identity on $(M_g, Y_g)$, we have
\begin{equation}\label{commutator for Lemma 4.2}
 [Y_g(u^1,x_1),Y_g(v^1,x_2)] \; = \; \Res_{x_0}\frac{1}{k}x_2^{-1}
\delta\Biggl(\frac{(x_1-x_0)^{1/k}}{x_2^{1/k}}\Biggr)Y_g(
Y(u^1,x_0)v^1,x_2).
\end{equation}
Therefore, 
\begin{eqnarray*}
\lefteqn{[Y_M(u,x_1),Y_M(v,x_2)]} \\
&=& [Y_g((\D(x^k_1)^{-1}u)^1,x_1^k),Y_g((\D(x_2^k)^{-1}v)^1,x_2^k)] \\
&=& \Res_{x}\frac{1}{k} x_2^{-k} \delta \left( 
\frac{(x_1^k - x)^{1/k}}{x_2} \right)  Y_g( Y( 
(\D(x_1^k)^{-1}u)^1,x) (\D(x_2^k)^{-1}v)^1,x_2^k).
\end{eqnarray*}

We want to make the change of variable $x=x_1^k-(x_1-x_0)^{k}$ where we
choose $x_0$ such that $((x_1 - x_0)^k)^{1/k} = x_1 - x_0$. Then noting
that $(x_1^k - x)^{n/k} |_{x = x_1^k-(x_1-x_0)^{k}} = (x_1 - x_0)^n$ for 
all $n \in \mathbb{Z}$, and using (\ref{residue change of variables}), 
we have
\begin{eqnarray*}
\lefteqn{[Y_M(u,x_1),Y_M(v,x_2)]} \\
&=& \Res_{x_0} x_2^{-k} (x_1-x_0)^{k-1} \delta \left( \frac{
x_1-x_0}{x_2} \right) Y_g( Y((\D(x_1^k)^{-1}u)^1,x_1^k -
(x_1-x_0)^{k})\\
& & \quad (\D(x_2^k)^{-1}v)^1,x_2^k)\\
&=& \Res_{x_0}  x_2^{-1}\delta\left(\frac{x_1-x_0}{x_2}\right) 
Y_g( Y( (\D(x_1^k)^{-1}u)^1, x_1^k - (x_1-x_0)^{k})\\
& & \quad (\D(x_2^k)^{-1}v)^1,x_2^k)\\
&=& \Res_{x_0}  x_1^{-1}\delta\left(\frac{x_2+x_0}{x_1}\right) 
Y_g( Y( (\D(x_1^k)^{-1}u)^1, x_1^k - (x_1-x_0)^{k})\\
& & \quad (\D(x_2^k)^{-1}v)^1,x_2^k)\\
&=& \Res_{x_0} x_1^{-1}\delta\left(\frac{x_2+x_0}{x_1}\right) \left(\frac{x_2+x_0}{x_1}\right)^{\frac{1}{2}(1-k)|u|} 
Y_g((Y(\D((x_2+x_0)^k)^{-1}u, \\
& & \quad (x_2+x_0)^k-x_2^{k})  \D(x_2^k)^{-1}v)^1,x_2^k)\\
&=& \Res_{x_0}  x_2^{-1} \delta \left( \frac{x_1-x_0}{x_2}\right) \left(\frac{x_1-x_0}{x_2}\right)^{\frac{1}{2}(k-1)|u|} 
Y_g((Y(\D((x_2+x_0)^k)^{-1}u, \\
& & \quad (x_2+x_0)^k-x_2^{k}) \D(x_2^k)^{-1}v)^1,x_2^k).
\end{eqnarray*}

Thus the proof is reduced to proving
\begin{eqnarray*} 
Y(\D((x_2+x_0)^k)^{-1}u,(x_2+x_0)^k - x_2^{k})\D(x_2^k)^{-1} = \D(x_2^k)^{-1} Y
\left(u, x_0 \right),
\end{eqnarray*} 
i.e., proving
\begin{equation}\label{3.10}
\D(x_2^k)Y(\D((x_2+x_0)^k)^{-1}u,(x_2+x_0)^k-
x_2^{k})\D(x_2^k)^{-1} =  Y \left(u,  x_0 \right).
\end{equation}
In Proposition \ref{psun1}, substituting $u$, $z$ and $z_0$ by
$ \D((x_2+x_0)^k)^{-1}u,$ $x_2^k$ and $(x_2+x_0)^k
- x_2^{k}$, respectively, gives equation (\ref{3.10}).
\end{proof}

\begin{thm}\label{t4.l} With the notations as above, for $k$ odd,  $U_g^k(M_g,Y_g) = 
(U_g^k(M_g),Y_M) = (M_g,Y_M)$ is a weak $V$-module.
\end{thm}

\begin{proof} Since the $L(-1)$-derivation property has been proved for $Y_M$ in 
Lemma \ref{l4.1}, we only need to prove the  Jacobi identity which is
equivalent to the supercommutator formula given by Lemma  \ref{l4.2}, since we have  restricted to the case when $k$ is odd, and the associator  formula which states that for $u,v\in V$ and $w
\in U_g^k(M_g)$ there exists  a positive integer $n$ such that
\[(x_0+x_2)^nY_M(u,x_0+x_2)Y_M(v,x_2)w=(x_2+x_0)^nY_M(Y(u,x_0)v,x_2)w .\]

Write $u^1=\sum_{j=0}^{k-1}u^1_{(j)}$ where $gu^1_{(j)} =
\eta^ju^1_{(j)}$.  Then {}from the twisted Jacobi identity, we have the
following associator: there exists a positive integer $m$ such that
for $n\geq m,$ and $j=0,...,k-1$,
\[(x_0+x_2)^{j/k+n}Y_g(u^1_{(j)},x_0+x_2)Y_g(v^1,x_2)w=
(x_2+x_0)^{j/ k+n}Y_g(Y(u^1_{(j)},x_0)v^1,x_2)w.\]

Note that $\Delta_k(x)^{-1} u \in (x^{\frac{1}{2k}})^{(k-1)2 \mathrm{wt} \, u}  V[x^{-1}]$.  
Thus for $k$ odd, we have that $\Delta_k(x^k)^{-1} u \in V[x, x^{-1}]$.  From this fact, the rest of the proof is analogous to the proof of Theorem 4.3 in \cite{BDM} where we use 
Proposition \ref{psun1} above instead of Proposition 2.2 in \cite{BDM}.  
\end{proof}

\begin{thm}\label{t4.ll} 
For $k$ an odd positive integer, and $g = (1 \; 2\; \cdots \; k)$, the map
$U_g^k$ is a functor {}from the category $\mathcal{ C}_w^g(\vtk)$ of
weak $g$-twisted $\vtk$-modules to the category $\mathcal{ C}_w(V)$ of
weak $V$-modules such that $T_g^k \circ U_g^k = 
id_{\mathcal{ C}_w^g(\vtk)}$ and $U_g^k \circ T_g^k = 
id_{\mathcal{ C}_w(V)}$.  In particular, the categories 
$\mathcal{ C}_w^g(\vtk)$ and $\mathcal{ C}_w(V)$ are isomorphic. Moreover,

(1) The restrictions of $T_g^k$ and $U_g^k$ to the category of
admissible  $V$-modules $\mathcal{ C}_a(V)$ and to the category of
admissible $g$-twisted $\vtk$-modules $\mathcal{ C}_a^g(\vtk)$, respectively, give category isomorphisms. In particular, $\vtk$ is $g$-rational if and only if $V$ is rational.

(2) The restrictions of $T_g^k$ and $U_g^k$ to the category of
ordinary $V$-modules $\mathcal{ C}(V)$ and to the category of ordinary
$g$-twisted $\vtk$-modules $\mathcal{ C}^g(\vtk)$, respectively, give
category isomorphisms.
\end{thm}

\begin{proof} It is trivial to verify $T_g^k \circ U_g^k = id_{\mathcal{
C}_w^g(\vtk)}$ and $U_g^k \circ T_g^k = id_{\mathcal{ C}_w(V)}$ {}from the
definitions of the functors $T_g^k$ and $U_g^k$.  Parts 1 and 2 follow {}from Theorem \ref{main1}.  \end{proof}

Using the functor $T_g^k$ giving the isomorphism between the categories $\mathcal{C}(V)$ and $\mathcal{C}(V^{\otimes k})$ as well as the actual construction of $g$-twisted $V^{\otimes k}$-modules from $V$-modules, we have a correspondence between graded traces of modules in $\mathcal{C}(V)$ and modules in $\mathcal{C}(V^{\otimes k})$ as detailed in the following corollary.

\begin{cor}\label{graded-dimension-corollary}  Let $g = (1 \; 2 \; \cdots \; k)$ for $k$ odd.  Then $(M, Y_M)$ is an ordinary $V$-module with graded dimension 
\[ \mathrm{dim}_q M =  tr_M  q^{-c/24 + L(0)} = q^{-c/24} \sum_{\lambda \in \mathbb{C}} \mathrm{dim} (M_\lambda) q^\lambda \]
if and only if $(T_g^k(M), Y_g)$ is an ordinary $(1 \; 2 \; \cdots \; k)$-twisted $V^{\otimes k}$-module with graded dimension
\[\mathrm{dim}_q T_g^k(M) =tr_{T_g^k(M)} q^{-c/24 + L^g(0)} = q^{ (k^2 - 1)c/24k} \mathrm{dim}_{q^{1/k}} M .\]
\end{cor}

\section{A counter example to the notion that Theorem \ref{t4.ll} extends to $k$-cycles of even length, and a conjecture}\label{counterexample-section}

We present a counter example originally given by the author along with Vander Werf in  \cite{BV-fermion} (following at times \cite{DZ2}) that shows that the notion that Theorem \ref{t4.ll} might extend to $k$-cycles of even length fails.  That is, not only must the construction be different than that for odd case, but that the classification in terms of $V$-modules is false.  Based on this example, in \cite{BV-fermion} we made a conjecture for the case when $k$ is even which we also present below.  

Let $V_{fer}$ be the one free fermion vertex operator superalgebra.  Using the construction in \cite{DZ2} the author along with Vander Werf showed that there is one irreducible $(1 \; 2 \; \cdots \; k)$-twisted $V_{fer}^{\otimes k}$-module for $k$ even, up to equivalence, and that this module splits into two parity-unstable invariant subspaces.  Since the sole irreducible $V_{fer}$-module, $V_{fer}$, does not split into invariant subspaces, the category of $(1 \; 2 \; \cdots \; k)$-twisted $V_{fer}^{\otimes k}$-modules can not be isomorphic to the category of $V_{fer}$-modules when $k$ is even.  This example shows that Theorem \ref{t4.ll}  fails in general for extensions to the symmetric group.

We further observe in \cite{BV-fermion}, that there is one unique, up to equivalence, parity-twisted $V_{fer}$-module and that this module splits into two parity-unstable invariant subspaces.  In addition, under the transformation $q \mapsto q^{1/k}$,  the graded dimension of the  unique, up to equivalence, irreducible $(1 \; 2 \; \cdots \; k)$-twisted $V_{fer}^{\otimes k}$-module is the same as the graded dimension of the irreducible parity-twisted $V_{fer}$-module. 

Based on these observations, in \cite{BV-fermion} we made the following conjecture.
\begin{conj}\label{conjecture}
The category of weak $(1 \; 2 \; \cdots \; k)$-twisted $V^{\otimes k}$-modules for $k$ even is isomorphic to the category of parity-twisted  $V$-modules.   Moreover, the subcategory of weak admissible or ordinary $(1 \; 2 \; \cdots \; k)$-twisted $V^{\otimes k}$-modules is isomorphic to the subcategory of weak admissible or ordinary parity-twisted $V$-modules, respectively.
\end{conj}  

The case of $V$ a vertex operator superalgebra and $V\otimes V$ being permuted by the $(1 \; 2)$ transposition is the mirror map if $V \otimes V$, in addition to being a vertex operator superalgebra is also N=2 supersymmetric (see for example, \cite{Barron-varna} and \cite{Barron-n2twisted}).  This is one of the motivations for wanting an extension of \cite{BDM} where all permutation twisted vertex operator algebras were constructed and classified, to the super setting, in particular for the even order permutation case.

\begin{rem}\label{NS-R-remark}
{\em Here we note some of the implications of this conjecture for the case of supersymmetric vertex operator superalgebras, i.e., those vertex operator superalgebras that, in addition to being a representation of the Virasoro algebra, also are representations of the Neveu-Schwarz algebra a Lie superalgebra extension of the Virasoro algebra.   See, e.g., \cite{Barron-announce}--\cite{Barron-n2twisted}.   In this case, when $V$ is a supersymmetric vertex operator superalgebra, the parity-twisted $V$-modules are natural a representation of the Ramond algebra.  This is another extension of the Virasoro algebra to a Lie superalgebra.  In physics terms, the modules for the supersymmetric vertex operator superalgebras are called the ``Neveu-Schwarz sectors" and the parity-twisted modules are called the ``Ramond sectors".  Thus the current work, i.e. Theorem \ref{t4.ll}, along with Conjecture \ref{conjecture}, if true, would imply that all permutation-twisted modules for tensor product supersymmetric vertex operator superalgebras are built up as tensor products of Neveu-Schwarz sectors (coming from the odd cycles) and Ramond sectors (coming from the even sectors).}
\end{rem}

\end{document}